\documentclass{amsart}

\usepackage{amsmath}
\usepackage{amsthm}
\usepackage{amsopn}
\usepackage{amssymb}
\usepackage{stmaryrd}
\usepackage[all]{xy}
\usepackage[nocolor]{sseq}
\usepackage{graphicx}
\usepackage{lscape}

\newcommand{\abs}[1]{\lvert #1 \rvert}

\newcommand{\br}[1]{\overline{#1}}

\newcommand{\td}[1]{\widetilde{#1}}
\newcommand{\tdd}[1]{\widetilde{\widetilde{#1}}}
\newcommand{\mmod}{\negmedspace \sslash \negmedspace}
\newcommand{\mc}[1]{\mathcal{#1}}
\newcommand{\ZZ}{\mathbb{Z}}

\newcommand{\FF}{\mathbb{F}}

\newcommand{\tmf}{\mathit{tmf}}

\newcommand{\ssb}{\bullet}
\newcommand{\ssc}{\circ}
\newcommand{\sscc}{\odot}
\newcommand{\ssccc}{\circledcirc}
\newcommand{\bxi}{\bar{\xi}}

\theoremstyle{definition}
 \newtheorem{thm}{Theorem}[section]
 \newtheorem{cor}[thm]{Corollary}
 \newtheorem{lem}[thm]{Lemma}
 \newtheorem{prop}[thm]{Proposition}

 \newtheorem{rmk}[thm]{Remark}
 \newtheorem{claim}[thm]{Claim}
\newtheorem*{thm*}{Theorem}
\newtheorem*{cor*}{Corollary}
\newtheorem*{lem*}{Lemma}
\newtheorem*{prop*}{Proposition}
\newtheorem*{defn*}{Definition}
\newtheorem*{ex*}{Example}
\newtheorem*{exs*}{Examples}
\newtheorem*{rmk*}{Remark}
\newtheorem*{claim*}{Claim}

\numberwithin{equation}{section}
\numberwithin{figure}{section}
\numberwithin{table}{section}
\DeclareMathOperator{\coker}{coker}
\DeclareMathOperator{\Ext}{Ext}
\DeclareMathOperator{\Hom}{Hom}

\DeclareMathOperator*{\Tot}{Tot}

\title{On the existence of a $v_2^{32}$-self map on $M(1,4)$ at the prime
$2$}
\author[M.~Behrens]{M.~Behrens$\sp 1$}
\address{
Dept. of Mathematics \\
M.I.T. \\
Cambridge, MA  02139, U.S.A.
}
\author{M.~Hill}
\address{
Dept. of Mathematics \\
University of Virginia \\
Charlottesville, VA 22904, U.S.A.
}
\author[M.J.~Hopkins]{M.J.~Hopkins$\sp 2$}
\address{
Dept. of Mathematics \\
Harvard University \\
Cambridge, MA  02138, U.S.A.
}
\author{M.~Mahowald}
\address{
Dept. of Mathematics \\
Northwestern University \\
Evanston, IL  60208, U.S.A.
}

\date{\today}

\begin{document}

\begin{abstract}
Let $M(1)$ be the mod $2$ Moore spectrum.  J.F.~Adams proved that $M(1)$
admits a minimal $v_1$-self map $v_1^4: \Sigma^8 M(1) \rightarrow M(1)$.
Let $M(1,4)$ be the cofiber of this self-map.  The purpose of this paper is
to prove that $M(1,4)$ admits a minimal $v_2$-self map of the form
$v_2^{32}: \Sigma^{192}M(1,4) \rightarrow M(1,4)$.  The existence of this
map implies the existence of many $192$-periodic families
of elements in the stable homotopy groups of spheres.
\end{abstract}

\footnotetext[1]{The author is supported by the NSF and the Sloan
Foundation.}
\footnotetext[2]{The author is supported by the NSF.}

\maketitle
\tableofcontents
\bibliographystyle{amsalpha}

\section{Introduction}

Fix a prime $p$.
The $p$-component of the 
stable homotopy groups of spheres admits a filtration called the
chromatic filtration.  
Elements in the $n$th layer of this filtration fit
into infinite $v_n$-periodic families.  Theoretically, this process is well
understood, thanks to the Nilpotence and Periodicity Theorems of Devinatz,
Hopkins, and Smith \cite{NilpI}, \cite{NilpII}.  

It is difficult
in practice, however, to explicitly identify $v_n$-periodic elements, and to
determine their periods.  One useful technique is to inductively form
cofiber sequences:
\begin{gather*}
S \xrightarrow{p^{i_0}} S \rightarrow M(i_0), \\
\Sigma^{2i_1(p-1)} M(i_0) \xrightarrow{v_1^{i_1}} M(i_0) \rightarrow M(i_0,i_1), \\
\vdots \\
\Sigma^{2i_n(p^n-1)} M(i_0, \ldots, i_{n-1}) \xrightarrow{v_n^{i_n}} M(i_0, 
\ldots, i_{n-1})
\rightarrow M(i_0, \ldots, i_n).
\end{gather*}
The maps $v_k^i$ are $v_k$-self maps.  The Periodicity Theorem guarantees
their existence for large $i$.
The reader is warned that there are potentially many non-homotopic
$v_k^i$-self maps, so the homotopy types of the spectra $M(i_0, \ldots,
i_n)$ are not determined merely from the sequence $(i_0, \ldots, i_n)$.

It is challenging to determine the
minimal sequence $(i_0, i_1, \ldots,
i_n)$.  This minimal sequence determines the periods of the primary
constituents of the
$v_n$-periodic families in the stable homotopy groups of spheres.  
We refer the reader to \cite[Ch.~5.5]{Ravenel},
\cite{Ravenelorange}, and \cite{Behrensroot2}
for a more detailed discussion.

We give a brief synopsis of what is known concerning the minimal sequence
of integers $(i_0, \ldots, i_n)$ so that the spectrum $M(i_0, \ldots, i_n)$
exists at a given prime $p$.
For $p \ge 3$, it is known that the complex $M(1,1)$ is minimal 
\cite{Adams}, 
for $p \ge 5$, the complex $M(1,1,1)$ is minimal \cite{Smith},
and for $p \ge 7$, the complex $M(1,1,1,1)$ is minimal
\cite{Toda}.  
For $p = 2$, the complex $M(1,4)$ is minimal \cite{Adams}, and for $p = 3$,
the complex $M(1,1,9)$ is minimal \cite{BehrensPemmaraju}.

In \cite{DavisMahowald}, it was argued that the complex $M(1,4,8)$ is minimal
at the prime $2$, i.e. that there is a $v_2$-self map:
$$ \Sigma^{48}M(1,4) \xrightarrow{v_2^8} M(1,4). $$
The result is incorrect: the image of $v_2^8$ in the Adams-Novikov spectral
sequence for $\tmf$ is not a permanent cycle \cite{HopkinsMahowald}, \cite{Bauer}.  In fact the
first multiple of $v_2$ which is a permanent cycle in this spectral
sequence is $v_2^{32}$.  The purpose of this paper is to prove the
following theorem.

\begin{thm}\label{thm:mainthm}
There is a $v_2^{32}$-self map
$$ v : \Sigma^{192} M(1,4) \rightarrow M(1,4). $$
\end{thm}

\begin{cor}
At the prime $2$, the complex $M(1,4,32)$ is minimal.
\end{cor}

\begin{rmk}
A $v_2^{32}$-self map is, by definition, a map $v$ whose induced map
$$ v_*: K(2)_*M(1,4) \rightarrow K(2)_*M(1,4) $$
is given by multiplication by $v_2^{32}$.  In particular, the map $v$, 
and all of
its iterates, must be essential.  
Since there is a map of ring spectra
$$ \tmf \rightarrow K(2) $$
under which the periodicity generator $v_2^{32} \in \pi_{192}(\tmf_2)$ maps
to $v_2^{32} \in \pi_{192}K(2)$,
to prove
Theorem~\ref{thm:mainthm}, it suffices to prove that there exists a
self-map $v$ such that
$$ v_*: \tmf_*M(1,4) \rightarrow \tmf_*M(1,4) $$
is given by multiplication by $v_2^{32}$.  
\end{rmk}

\begin{rmk} 
The fourth author reports that methods similar to those described in this 
paper
show that the spectra $A_1$ and $M(2,4)$ also admit $v_2^{32}$-self
maps.  Here, $A_1$ is a spectrum whose cohomology is a free module 
of rank $1$ over
the subalgebra $A(1)$ of the Steenrod algebra (see \cite{DavisMahowald}).
\end{rmk}

The self-map of Theorem~\ref{thm:mainthm} produces many $v_2^{32}$-periodic
infinite families of elements in the stable homotopy groups of spheres.
These families are discussed in detail in \cite{HopkinsMahowald}.  In fact,
all of the results of \cite{DavisMahowald} and \cite{Mahowaldv2} concerning
$v_2$-periodic families are valid with $v_2^8$ replaced by $v_2^{32}$.

\subsection*{Organization of the paper}

In Section~\ref{sec:Moore}, we reduce
Theorem~\ref{thm:mainthm} to showing that there exists a homotopy element
$$ v \in \pi_{192}(M(1,4) \wedge DM(1)) $$
with Hurewitz image $v_2^{32} \in \tmf_{192}(M(1,4) \wedge DM(1))$.
Here, $DM(1)$ is the Spanier-Whitehead dual of the spectrum $M(1)$.

In Section~\ref{sec:MASS} we construct modified Adams spectral
sequences (MASSs) of the form
\begin{align}
\Ext^{s,t}_{A_*}(\FF_2, H(1,4) \otimes DH(1,4)) & \Rightarrow
\pi_{t-s}(M(1,4) \wedge DM(1,4)), \label{eq:M14DM14} \\
\Ext^{s,t}_{A_*}(\FF_2, H(1,4) \otimes H_*(X)) & \Rightarrow
\pi_{t-s}(M(1,4) \wedge X) \label{eq:M14}
\end{align}
where $A_*$ is the dual Steenrod algebra, 
$H(1,4)$ and $DH(1,4)$ are objects in the derived category of 
$A_*$-comodules, 
and $\Ext_{A_*}$ is a group of homomorphisms in the derived category. 
We show that (\ref{eq:M14DM14}) is a spectral sequence of algebras, and
that (\ref{eq:M14}) is a spectral sequence of modules over
(\ref{eq:M14DM14}).

In Section~\ref{sec:v2_8} we prove that there exists an element
$$ v_2^8 \in \Ext^{8,56}_{A_*}(\FF_2, H(1,4) \otimes DH(1,4)). $$

In Section~\ref{sec:BrownGitler}, we give a general
overview of the theory of generalized 
Brown-Gitler $A_*$-comodules $M_i(j)$. We describe a
spectral sequence which computes $\Ext_{A_*}$ in terms of $\Ext_{A(i)}$ of
tensor products of these comodules.  The case of interest is where $i = 2$,
and the spectral sequence is an algebraic version of the $\tmf$-resolution.

In Section~\ref{sec:Charts} we compute
$$ \Ext^{*,*}_{A(2)_*}(H(1,4) \otimes M_2(1)^{\otimes k}) $$
for $k \le 3$.

In Section~\ref{sec:reduction} we establish vanishing lines for the Ext
groups appearing in the algebraic $\tmf$-resolution.  These vanishing lines
imply that the only targets of a potential differential supported by
$v_2^{32}$ are detected in the algebraic $\tmf$-resolution by the Ext groups
computed in Section~\ref{sec:Charts}.

In Section~\ref{sec:tmfM14}, we completely compute the MASS for $\tmf
\wedge M(1,4)$.

In Section~\ref{sec:d2d3}, we show that in the MASS for $M(1,4) \wedge
DM(1,4)$, the differential $d_2(v_2^8)$ is central.  This allows us to
deduce that $d_2(v_2^{16}) = 0$.  We then argue that the differential
$d_3(v_2^{16})$ is central, which implies that $d_3(v_2^{32}) = 0$.  We
just need to show that $v_2^{32}$ is a permanent cycle.

In Section~\ref{sec:diff}, we show that $\bar{\kappa}^6$ is killed in the
$E_3$-term of the MASS for $M(1,4) \wedge DM(1,4)$.

In Section~\ref{sec:mainthm}, we prove the main theorem.  We identify
possible targets of $d_r(v_2^{32})$ in the MASS for $M(1,4) \wedge DM(1)$
using the results of Sections~\ref{sec:Charts} and \ref{sec:reduction}, and
then eliminate these possibilities using the differentials computed in
Section~\ref{sec:tmfM14} and \ref{sec:diff}.

\subsection*{Conventions}

In this paper we shall always be implicitly working in the stable homotopy
category localized at the prime $2$.
All homology and cohomology groups in this paper are implicitly taken with
$\FF_2$ coefficients.

\subsection{Acknowledgments}

The authors would like to thank Robert R. Bruner for making his $\Ext$
software available.  The
first author completed some of this project while visiting Northwestern
University and Harvard University, and is grateful for the hospitality of
these institutions.  The authors also
express their thanks to the referee, for pointing out a gap in an
earlier version of this paper.

\section{Generalized Moore spectra}\label{sec:Moore}

Let $M(1)$ be the mod $2$ Moore spectrum.  There are many $v_1$-self-maps
$$ \Sigma^8 M(1) \rightarrow M(1), $$
however, low dimensional calculations indicate that there is precisely one with 
Adams filtration $4$.  We shall call this
map $v_1^4$, and its cofiber will be denoted $M(1,4)$.

It is useful to regard the desired self-map $v$ of Theorem~\ref{thm:mainthm} 
as an element of the homotopy group
$\pi_{192}(M(1,4) \wedge DM(1,4))$.
The proof of the theorem is simplified by the following splitting result.

\begin{prop}[Davis-Mahowald
{\cite[Lem~3.2]{DavisMahowald}}]\label{prop:DavisMahowald}
The projection
$$ M(1,4) \wedge DM(1,4) \rightarrow M(1,4) \wedge DM(1) $$
is a split surjection.
\end{prop}

\begin{cor}\label{cor:DavisMahowald}
An element $x \in \pi_{k}(M(1,4))$ extends to a self-map
$$ \td{x} : \Sigma^k M(1,4) \rightarrow M(1,4) $$
if and only if $2x = 0$.
\end{cor}

To prove Theorem~\ref{thm:mainthm}, it therefore suffices to construct an 
appropriate element 
$v' \in \pi_{192}(M(1,4) \wedge DM(1))$.

\section{Modified Adams spectral sequences}\label{sec:MASS}

For a graded Hopf algebra $\Gamma$ over a field $k$, let $\mc{D}_{\Gamma}$
denote the derived category of $\Gamma$-comodules.  
For objects $M$ and $N$
of $\mc{D}_\Gamma$, we define groups
$$ \Ext^{s,t}_{\Gamma}(M,N) = \mc{D}_{\Gamma}(\Sigma^t M, N[s]) $$
as a group of maps in the derived category.  Here $\Sigma^t M$ denotes the
$t$-fold shift with respect to the internal grading of $M$, and $N[s]$ denotes
the $s$-fold shift with respect to the triangulated structure of
$\mc{D}_\Gamma$.  This
reduces to the usual definition of $\Ext_\Gamma$ when $M$ and $N$ are
$\Gamma$-comodules.  We shall frequently use the abbreviation
$$ \Ext^{*,*}_{\Gamma}(M) := \Ext^{*,*}_{\Gamma}(k,M). $$
For a left $\Gamma$-comodule $M$ and a right $\Gamma$-comodule $N$, let
$C^*(N,\Gamma, M)$ denote the reduced cobar complex with
$$ C^s(N, \Gamma, M) = N \otimes \br{\Gamma}^{\otimes s} \otimes M. $$
Here, $\br{\Gamma}$ is the cokernel of the unit $k \rightarrow
\Gamma$.  Then $C^*(\Gamma, \Gamma, M)$ is an injective resolution for $M$
in the category of $\Gamma$-comodules, and
$$ \Ext^{s,t}_\Gamma(M) = H^s(C^*(k, \Gamma, M))_t. $$
We refer the reader to \cite[Appendix~1]{Ravenel} for details.

Let $A_*$ denote the dual Steenrod algebra.
Let $H(1) = H_*(M(1))$ be the homology of the mod $2$ Moore spectrum.
There is a triangle in $\mc{D}_{A_*}$:
\begin{equation}\label{eq:H(1)cofiber}
\Sigma \FF_2[-1] \xrightarrow{h_0} \FF_2 \rightarrow H(1) \rightarrow
\Sigma \FF_2.
\end{equation}
Let $v_1^4: \Sigma^{12} H(1)[-4] \rightarrow H(1)$ be the unique non-zero 
element of
$\Ext^{4,12}_{A_*}(H(1), H(1))$, which detects the $v_1$-self map of $M(1)$ in
Adams filtration $4$.
Let $H(1,4)$ denote the cofiber
\begin{equation*}\label{eq:H(1,4)cofiber}
\Sigma^{12} H(1)[-4] \xrightarrow{v_1^4} H(1) \rightarrow H(1,4)
\rightarrow \Sigma^{12} H(1)[-3].
\end{equation*}
Let 
$$ DM(1,4) = F(M(1,4), S) \simeq \Sigma^{-10} M(1,4) $$ 
denote the Spanier-Whitehead dual of $M(1,4)$, and let 
$$ DH(1,4) = \Hom_{\FF_2}(H(1,4), \FF_2) \cong \Sigma^{-13} H(1,4)[3] $$
denote the corresponding object in $\mc{D}_{A_*}$.

\begin{prop}\label{prop:MASS}
For a finite complex $X$, there are \emph{modified Adams spectral sequences}
(MASSs) of the form:
\begin{gather*}
E_2^{s,t}(M(1,4) \wedge X) = \Ext^{s,t}_{A_*}(H(1,4) \otimes 
H_*(X)) 
\Rightarrow
\pi_{t-s}(M(1,4)\wedge X),
\\
E_2^{s,t}(M(1,4) \wedge DM(1,4)) = 
\Ext^{s,t}_{A_*}(H(1,4) \otimes DH(1,4)) 
\Rightarrow
\pi_{t-s}(M(1,4) \wedge DM(1,4)).
\end{gather*}
\end{prop}

\begin{proof}
Consider the canonical Adams resolution of $M(1)$:
$$
\xymatrix{
M(1) \ar@{=}[r] & M(1)_0 \ar[d] & M(1)_1 \ar[l] \ar[d] & M(1)_2 \ar[d]
\ar[l] & \ar[l] \cdots \\
& K(1)_0 & K(1)_1 & K(1)_2 
}
$$
where
\begin{align*}
M(1)_i & = \br{H}^{\wedge i} \wedge M(1), \\
K(1)_i & = H \wedge \br{H}^{\wedge i} \wedge M(1).
\end{align*}
Here $H$ denotes the Eilenberg-MacLane spectrum $H\FF_2$, and $\br{H}$
denotes the fiber of the unit $S \rightarrow H$.
Since the self-map $v_1^4: \Sigma^8 M(1) \rightarrow M(1)$ has Adams
filtration $4$, there exists a lift:
$$
\xymatrix{
& M(1)_4 \ar[d] \\
\Sigma^8 M(1) \ar@{.>}[ru]^{\td{v_1^4}} \ar[r]_{v_1^4} & M(1)
}
$$
The lift $\td{v_1^4}$ induces a map of Adams resolutions:
\begin{equation}\label{diag:aresmap}
\xymatrix{
\Sigma^8 M(1)_0 \ar[d]_{v_1^4} & \ar@{=}[l] \cdots & \ar@{=}[l] \Sigma^8 M(1)_0
\ar[d]_{(\td{v_1^4})_0} & \Sigma^8 M(1)_1 \ar[l]
\ar[d]_{(\td{v_1^4})_1} & \cdots \ar[l] \\
M(1)_0 & \ar[l] \cdots & \ar[l] M(1)_4 & \ar[l] M(1)_5 & \cdots \ar[l]
}
\end{equation}
where the maps $(\td{v_1^4})_i$ are given by
$$ (\td{v_1^4})_i : \Sigma^8 M(1)_i = \Sigma^8 \br{H}^{\wedge i} \wedge
M(1) \xrightarrow{1 \wedge \td{v_1^4}} \br{H}^{\wedge i} \wedge
\br{H}^{\wedge 4} \wedge M(1) = M(1)_{i+4}.
$$
The mapping cones of the vertical maps of (\ref{diag:aresmap})
$$ \Sigma^8 M(1)_{i-4} \xrightarrow{(\td{v_1^4})_i} M(1)_i \rightarrow
M(1,4)_i $$
form a resolution:
$$
\xymatrix{
M(1,4) \ar@{=}[r] & M(1,4)_0 \ar[d] & M(1,4)_1 \ar[l] \ar[d] 
& \ar[l] M(1,4)_2 \ar[d] & \ar[l] \cdots \\
& K(1,4)_0 & K(1,4)_1 & K(1,4)_2 
}
$$
Smashing this resolution with $X$, we obtain a spectral sequence
\begin{equation}\label{eq:MASS1}
E_1^{s,t}(M(1,4) \wedge X) = \pi_{t-s}(K(1,4)_s \wedge X) \Rightarrow \pi_{t-s}(M(1,4)
\wedge X).
\end{equation}

By the $3 \times 3$ Lemma, the cofibers $K(1,4)_i$ fit into cofiber
sequences
\begin{equation}\label{eq:Kcofiber}
\Sigma^8 K(1)_{i-4} \xrightarrow{(\br{v_1^4})_i} K(1)_i \rightarrow
K(1,4)_i.
\end{equation}
Here we take $K(1)_{i-4} = \ast$ if $i < 4$, and $(\br{v_1^4})_i$ is the map
induced by smashing $(\td{v_1^4})_i$ with $H$.

Using the $A_*$-comodule structure of $H(1)$ together with the fact
that the composite
$$ S^8 \hookrightarrow \Sigma^8 M(1) \xrightarrow{v_1^4} M(1) $$
has Adams filtration $4$, one may easily check that the map
$$ \Sigma^{12} H(1) = \Sigma^{12} \pi_*K(1)_0 \xrightarrow{(\br{v_1^4})_0}
\Sigma^4 \pi_*K(1)_4 =
C^4(\FF_2, A_*, H(1)) $$
is injective.  It follows that the maps
$$ \Sigma^{12} C^{i-4}(\FF_2, A_*, H(1)) = \Sigma^{8+i}\pi_*K(1)_{i-4} 
\xrightarrow{(\br{v_1^4})_i}
\Sigma^{i} \pi_*K(1)_i = C^{i}(\FF_2, A_*, H(1)) $$
are injective for all $i$.  We conclude that the cofiber sequences
(\ref{eq:Kcofiber}) give rise to short exact sequences
\begin{multline*}
0 \rightarrow \Sigma^{12} C^{i-4}(\FF_2, A_*, H(1) \otimes H_*X)
\xrightarrow{(\br{v_1^4})_i} C^{i}(\FF_2, A_*, H(1) \otimes H_*X) 
\\
\rightarrow
\Sigma^i\pi_*(K(1,4)_i \wedge X) \rightarrow 0.
\end{multline*}
In the derived category $D_{A_*}$ we have a map of triangles:
$$
\xymatrix{
\Sigma^{12} C^{*-4}(A_*, A_*, H(1)) \ar[d]^{\simeq} \ar[r]^-{(\br{v_1^4})_*}
& C^*(A_*, A_*, H(1)) \ar[d]^\simeq \ar[r]
& Q(1,4)_* \ar[d]^\simeq
\\
\Sigma^{12}H(1)[-4] \ar[r]_{v_1^4} & H(1) \ar[r] & H(1,4) 
}
$$
where $Q(1,4)_i$ is the cokernel of the inclusion
$$ \Sigma^{12} C^{i-4}(A_*, A_*, H(1)) \xrightarrow{(\br{v_1^4})_i}
C^i(A_*, A_*, H(1)). $$
Since we have isomorphisms of cochain complexes
$$ \pi_{*}(K(1,4)_* \wedge X) \cong \Hom_{A_*}(\FF_2, Q(1,4)_* \otimes
H_*X), $$
we deduce that the $E_2$-term of the spectral sequence (\ref{eq:MASS1}) is
given by
$$ E_2^{s,t}(M(1,4) \wedge X) = \Ext^{s,t}_{A_*}(H(1,4)\otimes H_*X). $$

Consider the Adams resolution for the Spanier-Whitehead dual $DM(1)$:
$$
\xymatrix{
DM(1) \ar@{=}[r] & DM(1)_0 \ar[d] & DM(1)_1 \ar[l] \ar[d] & DM(1)_2 \ar[d]
\ar[l] & \ar[l] \cdots \\
& KD(1)_0 & KD(1)_1 & KD(1)_2 
}
$$
where
\begin{align*}
DM(1)_i & = F(M(1), \br{H}^{\wedge i}), \\
KD(1)_i & = F(M(1), H \wedge \br{H}^{\wedge i}).
\end{align*}
Define maps 
$(D\td{v_1^4})_i$ to be the composites
\begin{multline*}
(D\td{v_1^4})_i: DM(1)_i = F(M(1), \br{H}^{\wedge i}) \xrightarrow{u} 
F(\br{H}^{\wedge 4} \wedge M(1), \br{H}^{\wedge i + 4})
\\
\xrightarrow{(\td{v_1^4})^*} F(\Sigma^8 M(1),
\br{H}^{\wedge i+4}) = \Sigma^{-8} DM(1)_{i+4} 
\end{multline*}
where $u$ is the unit of the adjunction.
These maps assemble to give a map of Adams resolutions:
\begin{equation*}\label{eq:Dares}
\xymatrix{
DM(1)_0 \ar[d]_{Dv_1^4} & \ar@{=}[l] \cdots & \ar@{=}[l] DM(1)_0
\ar[d]_{(D\td{v_1^4})_0} & M(1)_1 \ar[l]
\ar[d]_{(D\td{v_1^4})_1} & \cdots \ar[l] \\
\Sigma^{-8} DM(1)_0 & \ar[l] \cdots & \ar[l] \Sigma^{-8} DM(1)_4 & 
\ar[l] \Sigma^{-8} DM(1)_5 & \cdots \ar[l]
}
\end{equation*}
Letting $DM(1,4)_i$ denote the homotopy fibers of the vertical maps of 
(\ref{eq:Dares}): 
$$ DM(1,4)_i \rightarrow DM(1)_i \xrightarrow{(D\td{v_1^4})_i} \Sigma^{-8}
DM(1)_{i+4}, $$
we obtain a modified Adams resolution of $DM(1,4)$:
$$
\xymatrix{
DM(1,4) \ar@{=}[r] & DM(1,4)_{-4} \ar[d] & DM(1,4)_{-3} \ar[l] \ar[d] &
DM(1,4)_{-2} \ar[d]
\ar[l] & \ar[l] \cdots \\
& KD(1,4)_{-4} & KD(1,4)_{-3} & KD(1,4)_{-2} 
}
$$
and a corresponding modified Adams spectral sequence
$$ E^{s,t}_2(DM(1,4)) = \Ext^{s,t}_{A_*}(DH(1,4)) \Rightarrow
\pi_{t-s}(DM(1,4)). $$

By taking iterated
mapping cylinders, we may assume that the maps
\begin{gather*}
M(1,4)_{i+1} \rightarrow M(1,4)_i, \\
DM(1,4)_{i+1} \rightarrow DM(1,4)_i
\end{gather*}
are inclusions of subcomplexes.
Taking the smash product of resolutions \cite[Ch.~IV, Def.~4.2]{Hinfty}
$$ \{ (M(1,4) \wedge DM(1,4))_i \} = \{ M(1,4)_i \} \wedge \{ DM(1,4)_i \} $$
gives the spectral sequence
$$ \Ext_{A_*}^{s,t}(H(1,4) \otimes DH(1,4)) \Rightarrow \pi_{t-s}(M(1,4) \wedge
DM(1,4)). $$
\end{proof}

\begin{prop}\label{prop:MASSmult}
The spectral sequence $\{ E_r(M(1,4) \wedge DM(1,4)) \}$ is a spectral
sequence of algebras, and the spectral sequence $\{ E_r(M(1,4) \wedge X)\}$
is a spectral sequence of modules over $\{ E_r(M(1,4) \wedge DM(1,4)) \}$.
\end{prop}

\begin{proof}
The canonical Adams resolution for the sphere spectrum is given by $\{
\br{H}^{\wedge i} \}$.  The canonical evaluation maps
\begin{multline*}
DM(1,4)_i \wedge M(1,4)_j = 
\\
(DM(1)_i \times_{(D\td{v_1^4})_i}
(\Sigma^{-8} DM(1)_{i+4})^I) \wedge (M(1)_j \cup_{(\td{v_1^4})_j} C\Sigma^8
M(1)_{j-4} \rightarrow \br{H}^{\wedge i+j}
\end{multline*}
induce maps of modified Adams resolutions
\begin{align*}
\{ (M(1,4) \wedge & DM(1,4))_i \} \wedge \{ (M(1,4) \wedge DM(1,4))_i \} 
\\
& = 
\{ M(1,4)_i \} \wedge \{ (DM(1,4) \wedge M(1,4))_i \} \wedge \{ DM(1,4)_i \}
\\
& \rightarrow 
\{ M(1,4)_i \} \wedge \{ \br{H}^{\wedge i} \} \wedge \{ DM(1,4)_i \}
\\
& = \{ (M(1,4) \wedge DM(1,4))_i \},
\end{align*}
\begin{align*}
\{ (M(1,4) \wedge & DM(1,4))_i \} \wedge \{ M(1,4)_i \wedge X \}
\\
& =
\{ M(1,4)_i \} \wedge \{ (DM(1,4) \wedge M(1,4))_i \} \wedge \{
\br{H}^{\wedge i} \wedge X \}
\\
& \rightarrow
\{ M(1,4)_i \} \wedge \{ \br{H}^{\wedge i} \} \wedge \{
\br{H}^{\wedge i} \wedge X \}
\\
& =
\{ M(1,4)_i \wedge X \}.
\end{align*}
These maps induce the desired pairings on the corresponding MASSs.
\end{proof}

\section{$v_2^8$-periodicity in $\Ext_{A_*}$}\label{sec:v2_8}

A similar (but easier) argument to Proposition~\ref{prop:DavisMahowald} 
proves the following lemma.

\begin{lem}\label{lem:algDavisMahowald}
The morphism
$$ H(1,4) \wedge DH(1,4) \rightarrow H(1,4) \wedge DH(1)
$$
is a split surjection.
\end{lem}

\begin{cor}\label{cor:algDavisMahowald}
An element $x \in \Ext_{A_*}^{s,t}(H(1,4))$ lifts to give an element
$\td{x}$ in $\Ext_{A_*}(H(1,4) \otimes DH(1,4))$ if an only if $h_0 x = 0$.
\end{cor}

A computation of $\Ext_{A(2)_*}(H(1,4))$ appears in
Figure~\ref{fig:tmfASS12}.  Note that it is
$v_2^8$-periodic.

\begin{prop}\label{prop:v_2^8}
There exists an element
$$ \td{v_2^8} \in \Ext_{A_*}^{8,56}(H(1,4)\otimes DH(1,4)) $$
which maps to the element $v_2^8 \in \Ext^{8,56}_{A(2)_*}(H(1,4))$ under the composite
$$ \Ext^{*,*}_{A_*}(H(1,4)\otimes DH(1,4)) 
\rightarrow \Ext^{*,*}_{A_*}(H(1,4)) \rightarrow \Ext_{A(2)_*}(H(1,4)). $$
\end{prop}

\begin{proof}
In the May spectral sequence for $\Ext_{A(2)_*}(\FF_2)$, the element
$v_2^{8}$ is detected by $b_{3,0}^4$.  Using Nakamura's formula
\cite{Nakamura}, and the calculations of \cite{Tangora}, we see that 
in the May spectral sequence for $\Ext_{A_*}(\FF_2)$, there
are differentials:
\begin{align*}
d_8(b_{3,0}^4) & = b_{2,0}^4 h_5, \\
d_4(b_{2,0}^2h_5) & = h_0^4 h_3 h_5.
\end{align*}
In the May spectral sequence, $v_1^4$ multiplication corresponds to
multiplication by $b_{2,0}^2$.
It follows that an element of $\Ext^{8,56}_{A_*}(H(1,4))$ which maps to
$$ v_2^8 \in \Ext^{8,56}_{A(2)_*}(H(1,4)) $$ 
must have image $h_0^3 h_3 h_5$
under the composite
$$ \Ext^{8,56}_{A_*}(H(1,4)) \xrightarrow{\delta_{v_1^4}} 
\Ext^{5,44}_{A_*}(H(1))
\xrightarrow{\delta_{v_0}} \Ext^{5,43}_{A_*}(\FF_2). $$
Since the element $h_0^3 h_3 h_5 \in \Ext^{5,43}_{A_*}(\FF_2)$ is killed by
$h_0$ multiplication, it lifts to an element $h_0^3 h_3 h_5[1] \in
\Ext^{5,44}_{A_*}(H(1))$.  Consider the exact sequence
$$ \Ext^{8,56}_{A_*}(H(1)) \rightarrow \Ext_{A_*}^{8,56}(H(1,4))
\rightarrow \Ext_{A_*}^{5,44}(H(1)) \xrightarrow{v_1^4}
\Ext^{9,56}_{A_*}(H(1)). $$
A computer calculation of $\Ext^{*,*}_{A_*}(H(1))$ using Bruner's programs
\cite{Bruner} reveals that:
\begin{enumerate}
\item
$\Ext^{9,56}_{A_*}(H(1)) = 0$,
\item
Every element $x \in \Ext^{5,44}_{A_*}(H(1))$ satisfies $h_0 x = 0$,
\item
Every element $y \in \Ext^{9,57}_{A_*}(H(1))$ satisfies $y = h_0 z$ for
some $z \in \Ext^{8,56}_{A_*}(H(1))$.
\end{enumerate}
These three facts allow us to deduce that there exists an element $w
\in \Ext^{8,56}_{A_*}(H(1,4))$ which maps to $h_0^3 h_3 h_5[1]$, and for
which we have $h_0 w = 0$.  By Corollary~\ref{cor:algDavisMahowald}, the
element $w$ lifts to the desired element $\td{v_2^8}$ in 
$\Ext^{8,56}_{A_*}(H(1,4)\otimes
DH(1,4))$.
\end{proof}

We shall abusively refer to the element $\td{v_2^8} \in
\Ext^{8,56}_{A_*}(H(1,4)\otimes DH(1,4))$ as $v_2^8$.

\section{Brown-Gitler comodules}\label{sec:BrownGitler}

\subsection*{Definitions}

Let $A(i)_*$ denote the quotient of the dual Steenrod algebra dual to the
subalgebra $A(i)$ of the Steenrod algebra.  There is an isomorphism
$$ A(i)_* \cong \FF_p[\bar{\xi}_1, \bar{\xi}_2, \bar{\xi}_3 \ldots,
\bar{\xi}_{i+1}]/(\bar{\xi}_1^{2^{i+1}}, \bar{\xi}^{2^i}_2,
\bar{\xi}_3^{2^{i-1}}, \ldots, \bar{\xi}_{i+1}^2 ). $$
Here, $\bar{\xi}_i$ denotes the conjugate of $\xi_i$.
We define a filtration on $A_\ast$
which induces a filtration on the $A_*$-subcomodule
$$
(A \mmod A(i))_* = 
A_\ast\Box_{A(i)_\ast}
\FF_2 \cong \FF_2[\bar{\xi}_1^{2^{i+1}},\bar{\xi}_2^{2^{i}},\ldots,
\bar{\xi}_{i+1}^2,\bar{\xi}_{i+2},\ldots ].
$$
Our filtration is an 
increasing filtration of algebras given on generators by 
$|\bar{\xi}_j|=2^{j-1}$. In particular, every element of 
$(A\mmod A(i))_*$ has filtration divisible by $2^{i+1}$.
The Brown-Gitler comodule $N_i(j)$ is the subspace of 
$(A\mmod A(i))_*$ spanned by
all elements of filtration less than or equal to $2^{i+1}j$.  
Using the coproduct formula
\begin{equation}\label{eq:coproduct}
\psi(\bar{\xi}_k) = \sum_{k_1 + k_2 = k} \bar{\xi}_{k_1} \otimes
\bar{\xi}_{k_2}^{2^{k_1}},
\end{equation}
the submodule
$N_i(j)$ is easily seen to be an $A_*$-subcomodule.  Thus we have an
increasing sequence of $A_*$-comodules:
$$ \FF_2 \cong N_i(0) \subset N_i(1) \subset N_i(2) \subset \cdots \subset
(A\mmod A(i))_*. $$

Define a map of ungraded rings
$$ \phi_i : (A \mmod A(i))_* \rightarrow (A \mmod A(i-1))_* $$
whose effect on generators is given by:
$$ \phi_i(\bar{\xi}_k^{2^l}) = 
\begin{cases}
\bar{\xi}_{k-1}^{2^l}, & k > 1, \\
1, & k = 1.
\end{cases}
$$

\begin{lem}\label{lem:phi_i}
The map $\phi_i$ is a map of ungraded $A(i)_*$-comodules.
\end{lem}

\begin{proof}
As an $A(i)_*$-comodule algebra, $(A \mmod A(i))_*$ is generated by
the elements $\{ \bar{\xi}_{1}^{2^{i+1}}, \bar{\xi}_2^{2^i}, \ldots \}$.
It therefore suffices to check that $\phi_i$ commutes with the coaction on
on these generators.  This is easily checked using the coproduct
formula~(\ref{eq:coproduct}) and the relations in $A(i)_*$.
\end{proof}

Let $M_i(j)$ denote the subspace of $(A \mmod A(i))_*$ spanned by the
monomials of filtration exactly $2^{i+1}j$.

\begin{lem}
The map $\phi_i$ maps the subspace $M_i(j)$ isomorphically onto the
$A_*$-subcomodule $N_{i-1}(j) \subset (A\mmod A(i-1))_*$.
\end{lem}

\begin{proof}
The subspace of $M_i(j)$ spanned by monomials of the form
$\bar{\xi}_1^{2^{i+1}s} x$
where $x$ is a monomial involving $\bar{\xi}_k^{2^l}$ for $k > 1$ is mapped
isomorphically onto the subspace $M_{i-1}(j-s) \subset N_{i-1}(j)$.  
\end{proof}

Using Lemma~\ref{lem:phi_i}, we have the following corollaries.

\begin{cor}
The subspace $M_i(j) \subset (A\mmod A(i))_*$ is an $A(i)_*$-subcomodule.
\end{cor}

\begin{cor}
There is an isomorphism of (graded) $A(i)_*$-comodules
$$ M_i(j) \cong \Sigma^{2^{i+1}j} N_{i-1}(j). $$
\end{cor}

\begin{cor}\label{cor:splitting}
There is a splitting of $A(i)_*$-comodules
$$ (A \mmod A(i))_* \cong \bigoplus_{j \ge 0} M_i(j). $$
\end{cor}

\begin{rmk}
The comodule $N_{-1}(j)$ (respectively $N_0(j)$, $N_1(j)$) is isomorphic 
as an
$A_*$-comodule to the homology of the $j$th $\ZZ/2$ (respectively integral,
$bo$) Brown-Gitler spectrum.  It is not known in general if the comodules
$N_i(j)$ are realizable for $i > 1$.
\end{rmk}

\subsection*{Algebraic resolutions}

We now describe an algebraic analog of an Adams resolution.  For $i = -1$
(respectively $i = 0,1,2$) this algebraic resolution will correspond to the 
$H\FF_2$ (respectively
$H\ZZ$, $bo$, $\tmf$) Adams resolution.

Let $X$ be an object of the derived category $\mc{D}_{A_*}$.  
We define $T_i(X)^\bullet$ to be the following cosimplicial object.
$$ 
\xymatrix@C+1.8em{
(A \mmod A(i))_* \otimes X 
\ar@<1ex>[r]|-{u \otimes 1} 
\ar@<-1ex>[r]|-{1 \otimes u}
& (A \mmod A(i))_*^{\otimes 2} \otimes X 
\ar@<2ex>[r]|-{u \otimes 1 \otimes 1} 
\ar@<0ex>[r]|-{1 \otimes u \otimes 1}
\ar@<-2ex>[r]|-{1 \otimes 1 \otimes u} 
& (A \mmod A(i))_*^{\otimes 3} \otimes X 
\ar@<3ex>[r] 
\ar@<1ex>[r]
\ar@<-1ex>[r] 
\ar@<-3ex>[r]
& \!
}
$$
Here, $u$ is the unit
$$ \FF_2 \rightarrow (A \mmod A(i))_*. $$
Since $(A \mmod A(i))_*$ is an algebra, the canonical map
$$ X \rightarrow \Tot( T^i(X)^\bullet ) $$
is a quasi-isomorphism (see, for instance, \cite[Prop.~8.6.8]{Weibel}).  
We therefore have a Bousfield-Kan spectral sequence
\begin{equation}\label{eq:BKSS}
E_1^{s,t,n} = \Ext^{s,t}_{A_*}((A \mmod A(i))_* \otimes 
\br{(A \mmod A(i))}_*^{\otimes n} \otimes X[-n])
\Rightarrow \Ext^{s,t}_{A_*}(X).
\end{equation}
where 
$$ \br{(A \mmod A(i))}_* = \coker \left( 
\FF_2 \xrightarrow{u} (A\mmod A(i))_* \right).
$$
The $E_1$-term can be simplified using a change of rings isomorphism,
together with the splitting of Corollary~\ref{cor:splitting}:
\begin{align*}
E_1^{s,t,n}  
& = \Ext^{s,t}_{A_*}((A \mmod A(i))_* \otimes 
\br{(A \mmod A(i))}^{\otimes n}_* \otimes X[-n]) \\
& \cong \Ext^{s,t}_{A(i)_*}(\br{(A \mmod A(i))}_*^{\otimes n} \otimes X[-n]) \\
& \cong 
\bigoplus_{j_1, \ldots, j_n \ge 1} 
\Ext^{s,t}_{A(i)_*}(M_i(j_1) \otimes \cdots \otimes M_i(j_n) \otimes X[-n]).
\end{align*}
We shall call this spectral sequence (\ref{eq:BKSS}) the 
\emph{$A \mmod A(i)$-resolution} for $X$.  
In this paper we are only be interested in the case where $i = 2$.  In this
case, we shall refer to the $A\mmod A(2)$-resolution as the \emph{algebraic
$\tmf$-resolution}.

\begin{lem}\label{lem:resolutionmult}
If $R$ is a monoid in the derived category $\mc{D}_{A_*}$, then the
$A\mmod A(i)$-resolution for $R$ is a spectral sequence of algebras. 
If $M$ is an $R$-module, then the $A\mmod A(i)$-resolution for $M$ 
is a spectral
sequence of modules over the $A\mmod A(i)$-resolution for $R$.
\end{lem}

\section{$\Ext$ computations}\label{sec:Charts}

In this section we describe $\Ext_{A(2)_*}^{s,t}(M)$ for various objects $M
\in \mc{D}_{A(2)_*}$.  We first explain the computations, and then describe
the methodology used to produce these computations.  
Charts displaying these $\Ext$ groups can be found
in the following figures:
\vspace{10pt}

\begin{tabular}{ll}
Figure~\ref{fig:A2-bo_1}: & $\Ext^{*,*}_{A(2)_*}(\FF_2)$ and
$\Ext^{*,*}_{A(2)_*}(M_2(1))$, \\
Figure~\ref{fig:bo_1bo_1-bo_1bo_1bo_1}: &
$\Ext^{*,*}_{A(2)_*}(M_2(1)^{\otimes 2})$ and
$\Ext^{*,*}_{A(2)_*}(M_2(1)^{\otimes 3})$, \\
Figure~\ref{fig:bo_1M1-bo_1bo_1M1}: & $\Ext^{*,*}_{A(2)_*}(M_2(1)\otimes
H(1))$ and
$\Ext^{*,*}_{A(2)_*}(M_2(1)^{\otimes 2} \otimes H(1))$, \\
Figure~\ref{fig:bo_1bo_1bo_1M1-bo_1M14}: &
$\Ext^{*,*}_{A(2)_*}(M_2(1)^{\otimes 3}\otimes
H(1))$ and
$\Ext^{*,*}_{A(2)_*}(M_2(1) \otimes H(1,4))$, \\
Figure~\ref{fig:bo_1bo_1M14-bo_1bo_1bo_1M14}: &
$\Ext^{*,*}_{A(2)_*}(M_2(1)^{\otimes 2}\otimes
H(1,4))$ and
$\Ext^{*,*}_{A(2)_*}(M_2(1)^{\otimes 3} \otimes H(1,4))$. \\
\end{tabular}
\vspace{10pt}

In each of these charts, the indexing has been modified to put the bottom
generator of $M_2(1)^{\otimes k}$ in internal degree $0$.  The meaning of
the notation in each of these charts is explained below.

\begin{figure}
\includegraphics[width=0.514\textwidth]{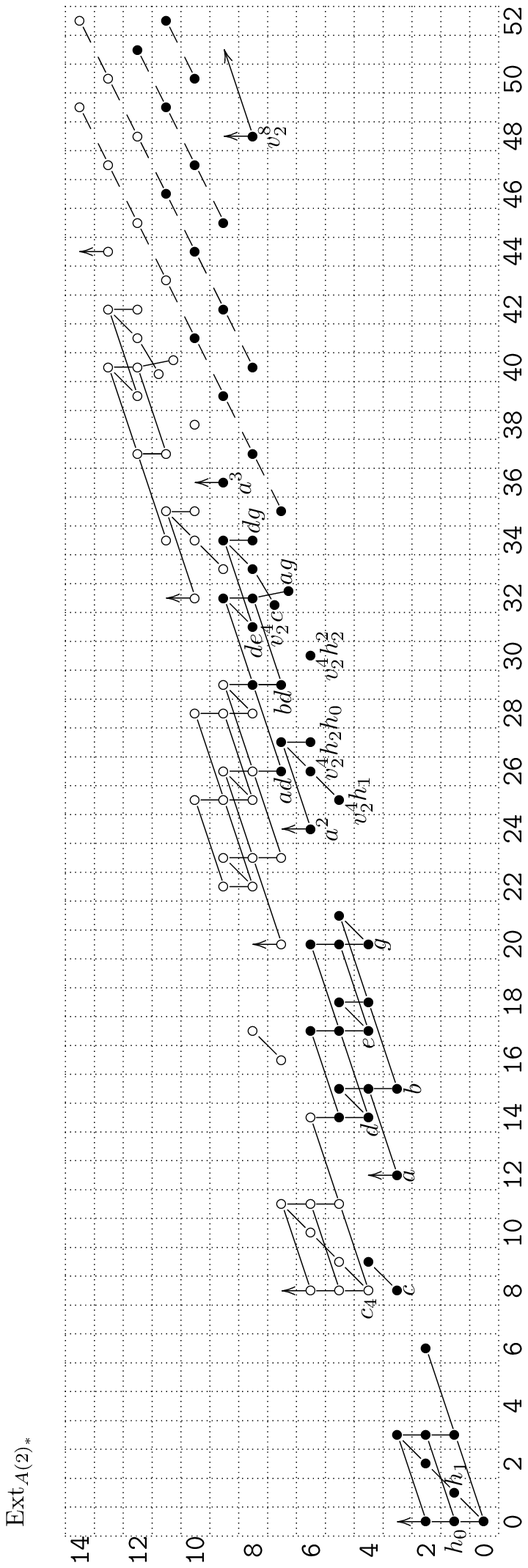}
\hfill
\includegraphics[width=0.446\textwidth]{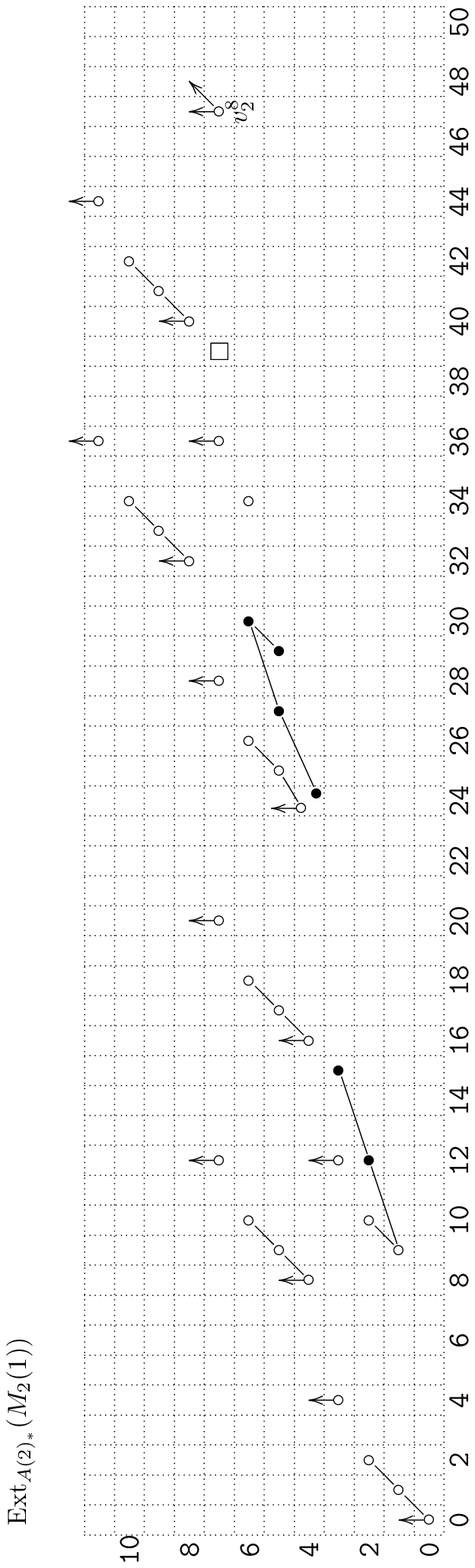}
\caption{}\label{fig:A2-bo_1}
\end{figure}

\begin{figure}
\includegraphics[width=0.488\textwidth]{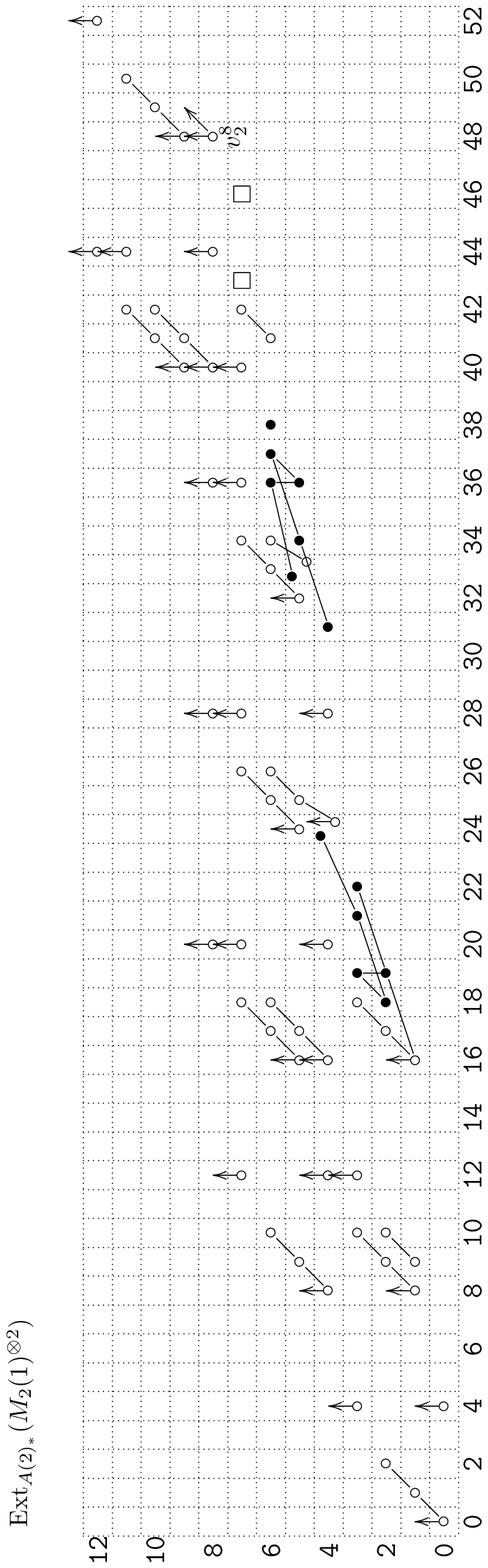}
\hfill
\includegraphics[width=0.472\textwidth]{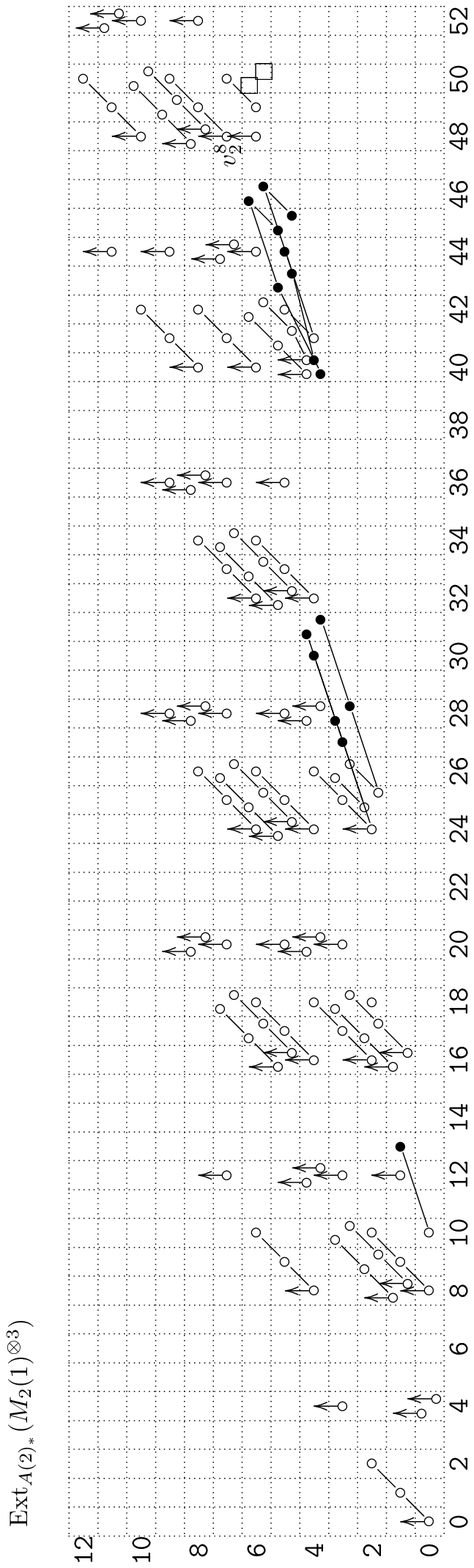}
\caption{}\label{fig:bo_1bo_1-bo_1bo_1bo_1}
\end{figure}

\begin{figure}
\includegraphics[width=0.460\textwidth]{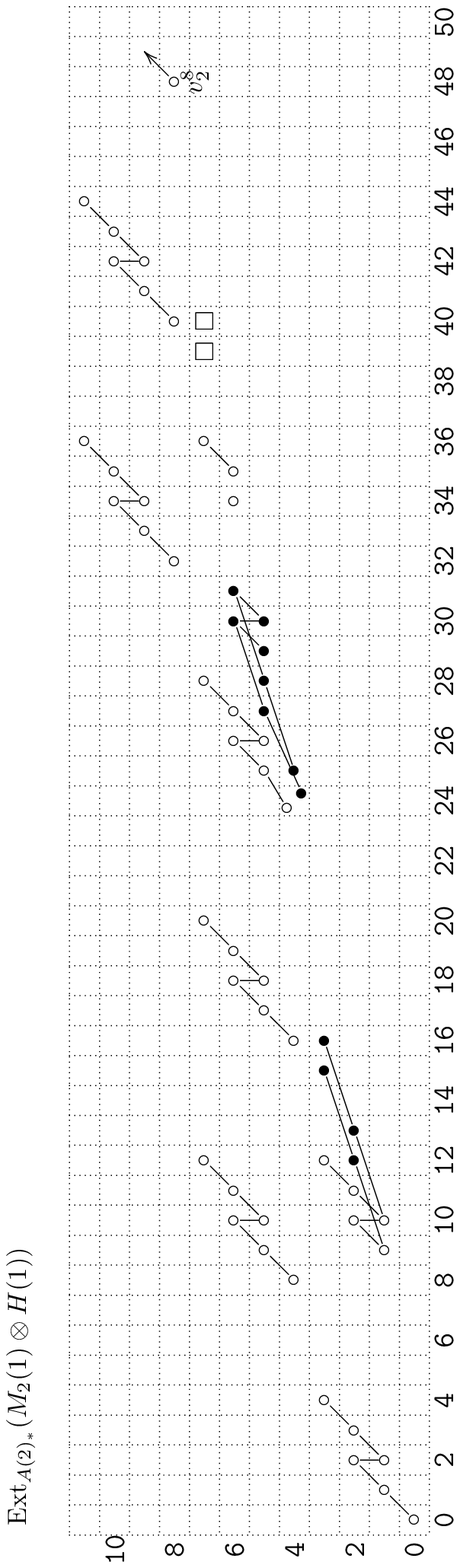}
\hfill
\includegraphics[width=0.500\textwidth]{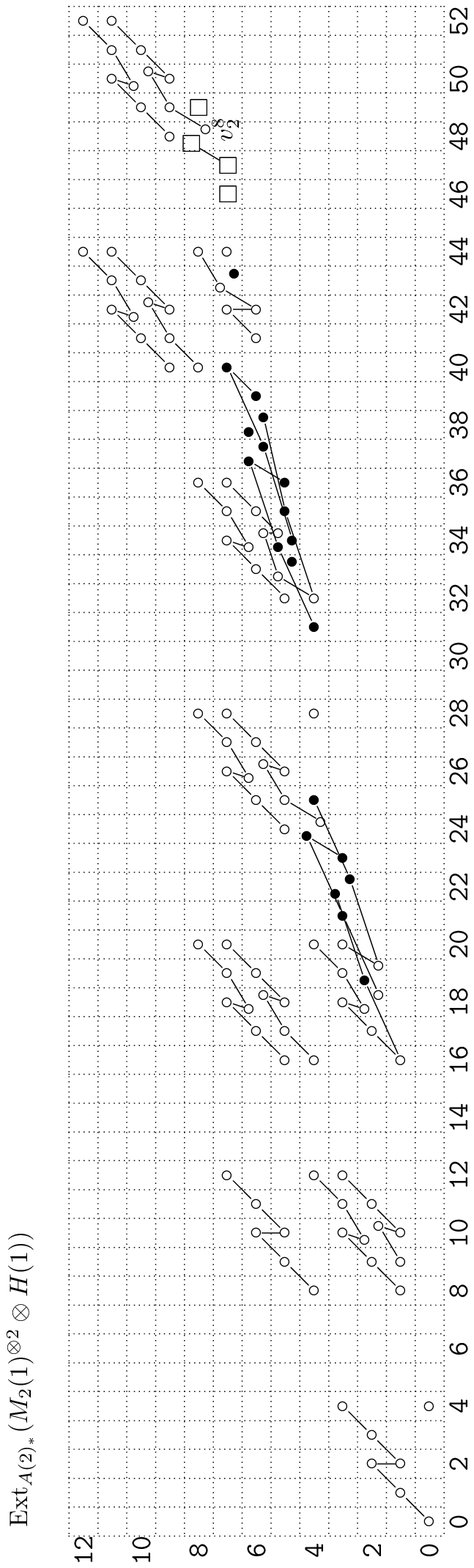}
\caption{}\label{fig:bo_1M1-bo_1bo_1M1}
\end{figure}

\begin{figure}
\includegraphics[width=0.520\textwidth]{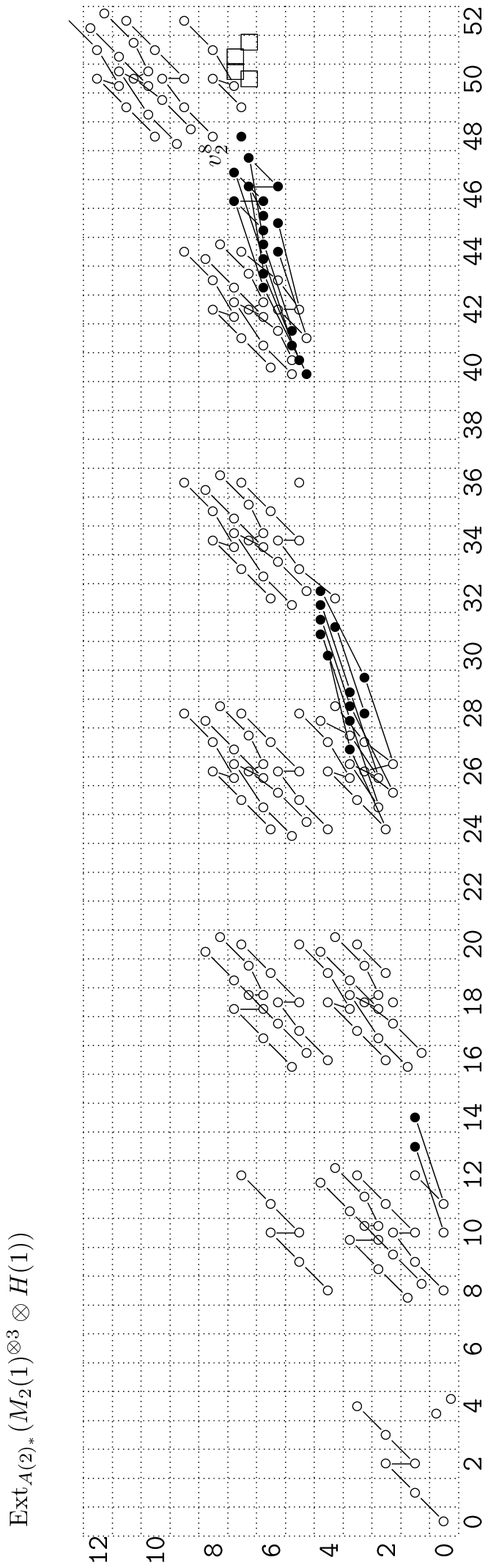}
\hfill
\includegraphics[width=0.440\textwidth]{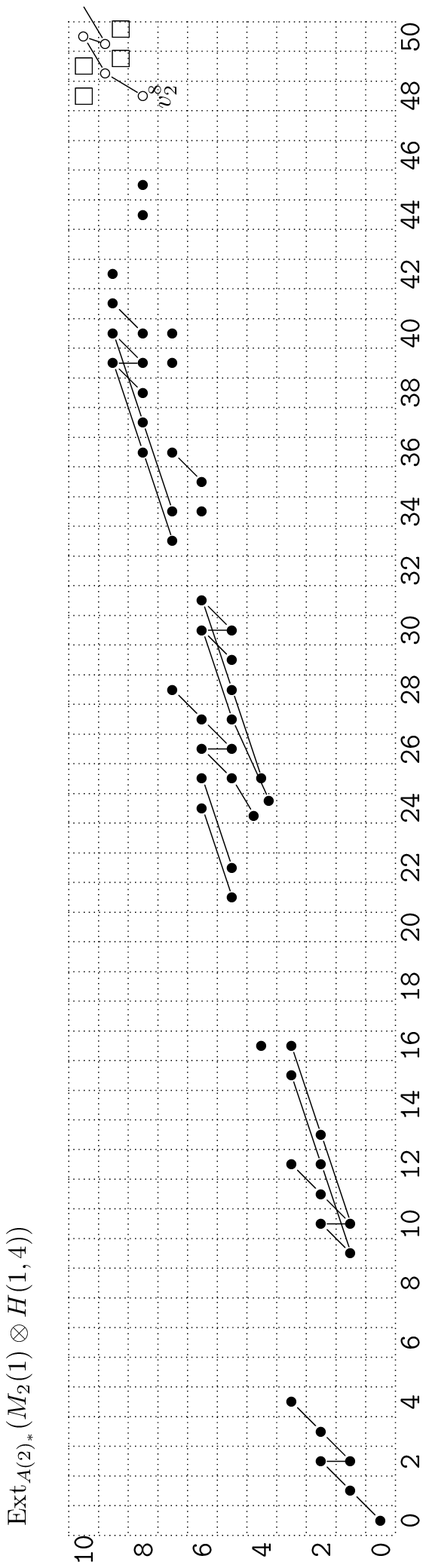}
\caption{}\label{fig:bo_1bo_1bo_1M1-bo_1M14}
\end{figure}

\begin{figure}
\includegraphics[width=0.37\textwidth]{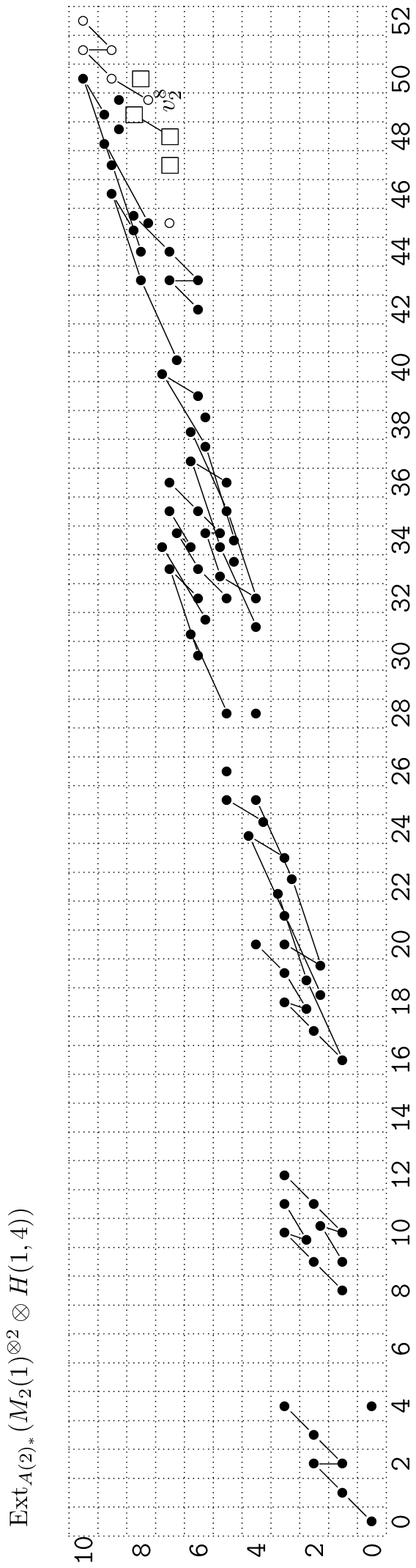}
\hfill
\includegraphics[width=0.43\textwidth]{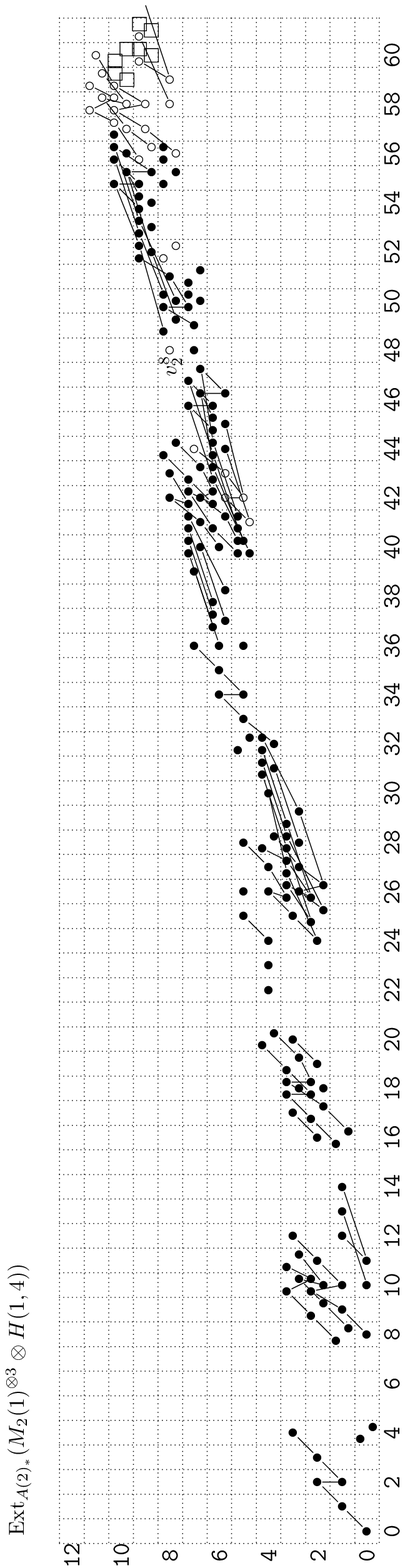}
\caption{}\label{fig:bo_1bo_1M14-bo_1bo_1bo_1M14}
\end{figure}

\subsection*{$\mathbf{Ext_{A(2)_*}(\FF_2)}$}
All of the elements are
$c_4 = v_1^4$-periodic, and $v_2^8$-periodic.  
Exactly one $v_1^4$ multiple
of each element is displayed with the $\bullet$ replaced by a $\circ$. 
Observe the wedge pattern beginning in $t-s = 35$.
This pattern is infinite, propagated horizontally by $h_{2,1}$-multiplication 
and vertically by $v_1$-multiplication.  Here, $h_{2,1}$ is the name of the
generator in the May spectral sequence of bidegree $(t-s,s) = (5,1)$, and
$h_{2,1}^4 = g$.

\subsection*{$\mathbf{Ext_{A(2)_*}(M_2(1)^{\otimes k})}$, for $\mathbf{k =
1,2,3}$}
Every element is $v_2^8$-periodic.
However, unlike $\Ext_{A(2)_*}(\FF_2)$, not every element of these Ext groups is
$v_1^4$-periodic.  Rather, it is the case that either an element $x \in
\Ext_{A(2)_*}(M_2(1)^{\otimes k})$ satisfies 
$v_1^4x = 0$, or it is $v_1^4$-periodic.
Each of the $v_1^4$-periodic 
elements fit into families which look like shifted and truncated copies of
$\Ext_{A(1)_*}(\FF_2)$, and are labeled with a $\circ$.  
We have only
included the beginning of these $v_1^4$-periodic patterns in the chart.  
The other
generators are labeled with a $\bullet$.  A $\Box$ indicates a
polynomial algebra $\FF_2[h_{2,1}]$.

\subsection*{$\mathbf{Ext_{A(2)_*}(M_2(1)^{\otimes k}\otimes H(1))}$, for
$\mathbf{k = 1,2,3}$}
The notation in these charts is identical to that in the charts for
$\Ext_{A(2)_*}(M_2(1)^{\otimes k})$, with the exception that the
$v_1^4$-periodic patterns are truncated shifted copies of
$\Ext_{A(1)_*}(H(1))$.

\subsection*{$\mathbf{Ext_{A(2)_*}(M_2(1)^{\otimes k}\otimes H(1,4))}$, for
$\mathbf{k = 1,2,3}$}
Because we have taken the cofiber of $v_1^4$, none of the elements are
$v_1^4$ periodic in these charts.  The generators of the first
$v_2^8$-periodic pattern are denoted with a $\bullet$ or a $\Box$, where
again a $\Box$ denotes a polynomial algebra on $h_{2,1}$.  In these charts, however, it
is not the case that every element is $v_2^8$-periodic: some elements in
the first lightening flash in the $0$-stem fail to be $v_2^8$-periodic.
We have conveyed this information by displaying the elements in the 
next $v_2^8$-pattern
with $\circ$.  With the exception of these first few generators, all of
the other generators are $v_2^8$-periodic.
\vfill

\subsection*{Methodology}

We explain how these charts were produced.  The computation of
$\Ext_{A(2)_*}(\FF_2)$ is well-known (see, for instance,
\cite{DavisMahowaldA(2)}).  The $A(2)_*$ 
comodule $M_2(1)$ can be described by the
following diagram of generators.
%\begin{equation}\label{diag:bo_1}
%\xymatrix@C-2em@R-2em{
%15 & \circ \ar@{-}[d] \\
%14 & \circ \ar@{-}@/^.5pc/[dd] \\
%\\
%12 & \circ \ar@{-} 	`/0pt[r]
%			`/0pt[rdddd]
%			[dddd]
%&
%\\
%\\
%\\
%\\
%8 & \circ &
%}
%\end{equation}
\begin{center}
\includegraphics{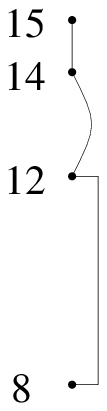}
\end{center}
Here, the dual action of the Steenrod algebra is encoded with a straight line
denoting $Sq^1_*$, a curved line denoting $Sq^2_*$, and the bracket
denoting $Sq^4_*$.  A computation of $\Ext_{A(2)_*}(M_2(1))$ can be found
in \cite{DavisMahowaldA(2)}.  The computation of
$\Ext_{A(2)_*}(M_2(1)^{\otimes 2})$ was obtained from $\Ext_{A(2)_*}(M_2(1))$ 
by inductively working up the skeletal
filtration of the second factor of $M_2(1)$.  The computation of 
$\Ext_{A(2)_*}(M_2(1)^{\otimes 3})$
was then obtained from $\Ext_{A(2)_*}(M_2(1)^{\otimes 2})$ 
by inductively working up the skeletal
filtration of the third factor of $M_2(1)$.  Along the way, because $H(1)$
occurs as a subcomodule of $M_2(1)$, we have computed
$\Ext_{A(2)_*}(M_2(1)^{\otimes k} \otimes H(1))$ for $k = 1,2$.  We then use
the long exact sequence induced by the triangle (\ref{eq:H(1)cofiber}) 
to obtain $\Ext_{A(2)_*}(M_2(1)^{\otimes 3} \otimes
H(1))$.

Each of these manual computations was independently verified by
R.R.~Bruner's computer program for computing $\Ext$ \cite{Bruner}.  
This computer program
constructs minimal resolutions of modules over the subalgebra $A(2)$.  We
also used the computer program to gain complete understanding of
$v_1^4$-periodicity in these $\Ext$ groups, as we now explain. Note
that there is an element
$$ v_1 \in \Ext^{1,3}_{A(2)_*}(H(1) \otimes H_*C\eta) $$
where $C\eta$ is the cofiber of $\eta \in \pi_1^s$.  We used Bruner's
programs to compute minimal resolutions for 
$$ \Ext_{A(2)_*}(M_2(1)^{\otimes k} \otimes H(1) \otimes H_*C\eta), \quad k =
1,2,3, $$
and read off all of the $v_1$-multiplicative structure in these $\Ext$
groups from the minimal resolutions.  We then used an $\eta$-Bockstein
spectral sequence to recover the $v_1^4$-multiplicative structure on 
$$ \Ext_{A(2)_*}(M_2(1)^{\otimes k} \otimes H(1)), \quad k = 1,2,3. $$
From this, the computation of
$$ \Ext_{A(2)_*}(M_2(1)^{\otimes k} \otimes H(1,4)), \quad k = 1,2,3 $$
was easily determined by the long exact sequence arising from the triangle
(\ref{eq:H(1,4)cofiber}).

\section{Reducing the computation to $M_2(1)^{\otimes k}$ for $k \le
3$}\label{sec:reduction}

\subsection*{Inductive Short Exact Sequences}

We will construct some short exact sequences that relate
the various Brown-Gitler comodules $N_1(j)$.  
We have an isomorphism
$$ (A(2)\mmod A(1))_* \cong \Lambda[\bar{\xi}_1^4, \bar{\xi}_2^2,
\bar{\xi}_3]. $$
Observe that there is an
isomorphism of $\FF_2$-vector spaces
$$ \tau : (A \mmod A(1))_* \xrightarrow{\cong} (A \mmod A(2))_* \otimes
(A(2) \mmod A(1))_* $$
given on the monomial basis by
$$ \tau(\bar{\xi}_1^{8i_1 + 4\epsilon_1} \bar{\xi}_2^{4i_2 + 2 \epsilon_2} 
\bar{\xi}_3^{2i_3 + \epsilon_3} \bar{\xi}_4^{i_4} \cdots )
= \bar{\xi}_1^{8i_1} \bar{\xi}_2^{4i_2}\bar{\xi}_3^{2i_3} \bar{\xi}_4^{i_4} \cdots
\otimes \bar{\xi}_1^{4\epsilon_1} \bar{\xi}_2^{2\epsilon_2}
\bar{\xi}_3^{\epsilon_3} $$
for $i_j \ge 0$ and $\epsilon_j = 0,1$.
The map $\tau$ is \emph{not} an isomorphism of $A(2)_*$-comodules.  For
instance, in $(A \mmod A(1))_*$
we have the coaction
$$ \psi(\bar{\xi}_1^4 \bar{\xi}_2^2) = \bxi_1^4 \bxi_2^2 \otimes 1 + \bxi_1^4
\otimes \bxi_2^2 + \bxi_2^2 \otimes \bxi_1^4 + 1 \otimes \bxi_1^4 \bxi_2^2
+ \bxi_1^6 \otimes \bxi_1^4 + \bxi_1^2 \otimes \bxi_1^8 $$
whereas in $(A\mmod A(2))_* \otimes (A \mmod A(1))_*$ we have
$$ \psi(1 \otimes \bxi_1^4 \bxi_2^2) = \bxi_1^4 \bxi_2^2 \otimes 1 \otimes
1 + \bxi_1^4
\otimes 1 \otimes \bxi_2^2 + \bxi_2^2 \otimes 1 \otimes \bxi_1^4 + 1
\otimes 1 \otimes \bxi_1^4 \bxi_2^2
+ \bxi_1^6 \otimes 1 \otimes \bxi_1^4. $$
However, there is a decreasing filtration 
$$ (A \mmod A(1))_* = F^0 (A \mmod A(1))_* \supset F^1 (A \mmod A(1))_*
\supset \cdots $$ 
of $A(2)_*$-comodules such that
$\tau$ induces an isomorphism of the associated graded $A(2)_*$-comodules
$$ \tau: 
E^0 (A\mmod A(1))_* \xrightarrow{\cong} (A \mmod A(2))_* \otimes (A(2)
\mmod A(1))_*. $$
The decreasing filtration is given as follows: under the isomorphism
$$ (A \mmod A(2))_* \cong \bigoplus_k M_2(k) $$
of $A(2)_*$-comodules given by Corollary~\ref{cor:splitting}, we define
$$ F^j(A \mmod A(1))_* := \tau^{-1}\left(\Big( \bigoplus_{k = j}^\infty 
M_2(k)\Big) \otimes
(A(2) \mmod A(1))_* \right). $$
Using the coproduct formula (\ref{eq:coproduct}) this is easily verified to
be a decreasing filtration by $A(2)_*$-comodules --- the coaction
preserves or raises the filtration.

Consider the quotients
$$ Q^j (A \mmod A(1))_* := (A \mmod A(1))_* / F^{j+1} (A \mmod A(1))_*. $$
The map $\tau$ induces isomorphisms of $\FF_2$-vector spaces
$$ \tau: Q^j (A \mmod A(1))_* \xrightarrow{\cong} N_2(j) \otimes (A(2) \mmod
A(1))_*. $$
Furthermore, the filtration $\{ F^k (A \mmod A(1))_* \}$ projects to
a finite decreasing filtration of $Q^j(A \mmod A(1))_*$ by
$A(2)_*$-comodules, such that $\tau$ induces an isomorphism of associated
graded $A(2)_*$-comodules
\begin{equation}\label{eq:E0tau}
\tau : E^0 Q^j(A \mmod A(1))_* \xrightarrow{\cong} N_2(j) \otimes (A(2)
\mmod A(1))_*. 
\end{equation}

\begin{lem}\label{lem:OddSeq}
There is a short exact sequence of $A(2)_*$-comodules:
$$
0 \rightarrow  \Sigma^{8j} N_1(j) \otimes N_1(1)
\rightarrow 
N_1(2j+1) \rightarrow 
Q^{j-1}(A\mmod A(1))_*
 \rightarrow 0.
$$
\end{lem}

\begin{lem}\label{lem:EvenSeq}
There is an exact sequence of $A(2)_*$-comodules:
$$
0
\rightarrow \Sigma^{8j} N_1(j)
\rightarrow N_1(2j)
\rightarrow Q^{j-1}(A \mmod A(1))_*
\rightarrow \Sigma^{8j+9} N_1(j-1) 
\rightarrow 0.
$$
\end{lem}

\begin{proof}[Proof of Lemma~\ref{lem:OddSeq}]
Since the elements of $(A(2) \mmod A(1))_*$ have Brown-Gitler 
filtration at most $12$,
the image of the composite
$$ N_2(j-1) \otimes (A(2)\mmod A(1))_* \hookrightarrow (A \mmod A(2))_*
\otimes (A(2) \mmod A(1))_* \xrightarrow{\tau^{-1}} (A \mmod A(1))_* $$
lies in $N_1(2j+1)$, giving a surjection of $A(2)_*$-comodules
$$ \rho : N_1(2j+1)
\twoheadrightarrow
Q^{j-1} (A \mmod A(1))_*. $$
As $\FF_2$-vector spaces, we have
$$ \tau(N_1(2j+1)) = N_2(j-1) \otimes (A(2) \mmod A(1))_* \oplus M_2(j)
\otimes N_1(1) $$
where the Brown-Gitler comodule $N_1(1)$ is identified as the $A(2)_*$-subcomodule
$$ N_1(1) = \FF_2\{ 1, \bar{\xi}_1^4, \bar{\xi}_2^2, \bar{\xi}_3 \} \subset
(A(2)\mmod A(1))_*.
$$
We deduce that the kernel of $\rho$ is 
$$ M_2(j) \otimes N_1(1) \cong \Sigma^{8j} N_1(j) \otimes
N_1(1). $$
\end{proof}

\begin{proof}[Proof of Lemma~\ref{lem:EvenSeq}]
As an $\FF_2$-vector space, the image of $N_1(2j)$ in $(A \mmod A(2))_*
\otimes (A(2)\mmod A(1))_*$ under the isomorphism $\tau$ is 
given by
$$ \tau(N_1(2j)) \cong 
\left( \begin{array}{c}
N_2(j-2) \otimes (A(2)\mmod A(1))_* \\
\oplus \\
M_2(j-1) \otimes \FF_2\{ 1, \bar{\xi}_1^4, \bar{\xi}_2^2, \bar{\xi}_3, 
\bar{\xi}_1^4\bar{\xi}_2^2,\bar{\xi}_1^4\bar{\xi}_3, 
\bar{\xi}_2^2\bar{\xi}_3 \} \\
\oplus \\
M_2(j) \otimes \FF_2\{1\}.
\end{array} \right)
$$
Thus, at least on the level of $\FF_2$-vector spaces, we have an exact
sequence
\begin{multline*}
0 \rightarrow M_2(j) \otimes \FF_2\{1\} 
\xrightarrow{\alpha} N_1(2j) 
\xrightarrow{\beta} Q^{j-1} (A \mmod A(1))_*
\\
\xrightarrow{\gamma} M_2(j-1)\otimes \FF_2\{\bar{\xi}_1^4\bar{\xi}_2^2\bar{\xi}_3\}
\rightarrow 0
\end{multline*}
We just need to prove that these are maps of $A(2)_*$-comodules.
The map $\gamma$ is clearly a map of $A(2)_*$-comodules.
We have the following diagram of inclusions of $A(2)_*$-comodules.
\begin{equation}\label{diag:inclusions}
\xymatrix{
M_2(j) \otimes \FF_2\{1\} \ar@{^{(}->}[r]^-{\alpha} \ar@{^{(}->}[d] &
N_1(2j) \ar@{^{(}->}[r]^-{\delta} \ar@{^{(}->}[d] &
Q^j (A\mmod A(1))_*  
\\
(A \mmod A(2))_* \ar@{^{(}->}[r] &
(A \mmod A(1))_* \ar@{->>}[ru] 
}
\end{equation}
In particular, the map $\alpha$ is a map of $A(2)_*$-comodules.
Let $K$ be the cokernel of $\alpha$.  Then we get an induced map of short
exact sequences of $A(2)_*$-comodules:
$$
\xymatrix@C-0.6em{
M_2(j)\otimes \FF_2\{1\} \ar[r]^\alpha \ar[d] &
N_1(2j) \ar[r]^{\beta_1} \ar[d]_\delta &
K \ar@{.>}[d]^{\beta_2} 
\\
M_2(j) \otimes (A(2)\mmod A(1))_* \ar[r] &
Q^j(A\mmod A(1))_* \ar[r] &
Q^{j-1}(A\mmod A(1))_* 
}
$$
We deduce that the map $\beta$ is a map of $A(2)_*$-comodules, because it is
given by the composite $\beta_2 \circ \beta_1$ of $A(2)_*$-comodule maps.
\end{proof}

\subsection*{Vanishing lines}

We reduce the computations needed to those of $M_2(1)^{\otimes k}$ for $k
\le 3$
using vanishing lines for modified Adams $E_2$ terms. 
Note that after a finite range, $\Ext_{A(2)_*}(H(1,4))$ has a vanishing
line of slope $1/5$.

\begin{lem}\label{lem:vanishinglines}
We have
$$ \Ext^{s,t}_{A(2)_*}(N_1(j) \otimes H(1,4)) = 0 $$
for 
$$ s > \max \left\{ \frac{(t-s)+17}{7}, \frac{(t-s) + a_j}{6}, \frac{(t-s) + b_j}{5} \right\} $$
and the constants $a_j$ and $b_j$ are inductively defined by
\begin{align*}
a_0 & = 21, \\
b_0 & = 9, \\
a_1 & = 15, \\
b_1 & = 2, \\
a_{2j} & = \max \{ a_{j-1}-8j-2, a_j-8j \}, \\
b_{2j} & = \max \{ b_{j-1}-8j-3, b_j-8j \}, \\
a_{2j+1} & = a_j-8j, \\
b_{2j+1} & = b_j-8j.
\end{align*}
\end{lem}

\begin{proof}
The case of $j = 0,1$ is obtained by examining 
Figures~\ref{fig:bo_1bo_1bo_1M1-bo_1M14} and \ref{fig:tmfASS12}.  The case
of $j \ge 2$ is established by induction using Lemmas~\ref{lem:OddSeq} and
\ref{lem:EvenSeq}.
The terms involving $Q^j(A \mmod A(1))_*$ are handled using the spectral
sequence
$$ \Ext^{s,t}_{A(2)_*}(N_2(j) \otimes (A(2) \mmod A(1))_* \otimes H(1,4)) 
\Rightarrow \Ext^{s,t}_{A(2)_*}(Q^j(A \mmod A(1))_* \otimes
H(1,4)) $$
induced from (\ref{eq:E0tau}),
and the change-of-rings isomorphism
$$
\Ext^{s,t}_{A(2)_*}(N_2(j) \otimes (A(2)\mmod A(1))_* \otimes H(1,4)) 
\cong \Ext^{s,t}_{A(1)_*}(N_2(j) \otimes H(1,4)).
$$
The only non-zero values values of $\Ext^{s,t}_{A(1)_*}(H(1,4))$ are
displayed below.
\begin{center}
\includegraphics{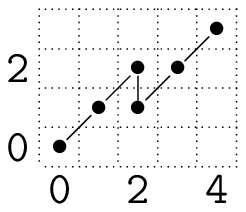}
\end{center}
In particular, we see that $\Ext^{s,t}_{A(1)_*}(H(1,4))$ is zero for $s >
\frac{(t-s)+17}{7}$.
\end{proof}

We extract the following estimate.

\begin{lem}\label{lem:vanishing1}
Suppose that $j_1, \ldots, j_n$ is a sequence of positive integers such
that for some $i$, $j_i \ge 2$.  Then we have
$$ \Ext^{s,t}_{A(2)_*}(M_2(j_1) \otimes \cdots \otimes M_2(j_n)[-n] \otimes H(1,4)) = 0 $$
for
$$
s > 
\max \left\{ \frac{(t-s)+17}{7}, \frac{(t-s) + 2}{6}, \frac{(t-s) - 12}{5}
\right\}.
$$
\end{lem}

\begin{proof}
Assume that $n = 1$, and set $j$ equal to  $j_1 \ge 2$.  
By Lemma~\ref{lem:vanishinglines}, we have
$$ \Ext_{A(2)_*}(M_2(j)[-1] \otimes H(1,4)) = 0 $$
if
$$ s > \max \left\{ \frac{(t-s)-8j+8+17}{7}, \frac{(t-s)-8j+7+a_j}{6}, \frac{(t-s)-8j+6+b_j}{5}
\right\}. $$
It therefore suffices to prove that the following inequalities are
satisfied:
\begin{align*}
17 & \ge 17 - 8j + 8, \\
2 & \ge a_j - 8j + 7, \\
-12 & \ge b_j - 8j + 6.
\end{align*}
The inequalities are true for $j = 2,3$.
By induction, these inequalities hold for all $j$.

We now induct on $n$.  We may as well assume that $j_1 \ge 2$.  Assume that
$$ \Ext^{s,t}_{A(2)_*}(M_2(j_1) \otimes \cdots \otimes M_2(j_{n-1}) [-n+1] \otimes H(1,4))
= 0 $$
for 
$$
s > 
\max \left\{ \frac{(t-s)+17}{7}, \frac{(t-s) + 2}{6}, \frac{(t-s) - 12}{5}
\right\}.
$$
By filtering the $A_*$-comodule $M_2(j_n)$
by degree, we obtain an Atiyah-Hirzebruch type spectral sequence which
converges to  
$$
\Ext^{s,t}_{A(2)_*}(M_2(j_1) \otimes \cdots \otimes M_2(j_n)[-n] \otimes H(1,4))
$$
and whose $E_1$-page is given by
$$
\bigoplus_{x} \Ext^{s,t}_{A(2)_*} (\Sigma^{\abs{x}} M_2(j_1) \otimes \cdots 
\otimes M_2(j_{n-1})[-n] \otimes H(1,4)).
$$
where $x$ ranges over an $\FF_2$-basis of $M_2(j_n)$.  The smallest value
$\abs{x}$ can take is $8$, in the case $j_1 = 1$.
By our inductive hypothesis, we have
$$
\Ext^{s,t}_{A(2)_*} (\Sigma^{8} M_2(j_1) \otimes \cdots 
\otimes M_2(j_{n-1})[-n] \otimes H(1,4)) = 0
$$
for
$$
s > 
\max \left\{ \frac{(t-s)+17}{7}, \frac{(t-s) + 1}{6}, \frac{(t-s) - 14}{5}
\right\}.
$$
This verifies the inductive step.
\end{proof}

\begin{lem}\label{lem:vanishing2}
Suppose that $n$ is greater than $3$.  Then we have
$$ \Ext^{s,t}_{A(2)_*}(M_2(1)^{\otimes n}[-n] \otimes H(1,4)) = 0 $$
for
$$ s > \max \left\{ \frac{(t-s)+17}{7}, \frac{(t-s)+4}{6}, \frac{(t-s)-17}{5} 
\right\}. $$
\end{lem}

\begin{proof}
Examining Figure~\ref{fig:bo_1bo_1M14-bo_1bo_1bo_1M14}, we see that
$$
\Ext^{s,t}_{A(2)_*}(N_1(1)^{\otimes 3} \otimes H(1,4)) = 0
$$
for
$$ s > \max \left\{ 
\frac{(t-s)+17}{7}, \frac{(t-s)+8}{6}, \frac{(t-s)-9}{5} \right\}. $$
The lemma follows from induction on $n$, using Atiyah-Hirzebruch type
spectral sequences as in the proof of Lemma~\ref{lem:vanishing2}.
\end{proof}

\section{The modified Adams spectral sequence for
$\tmf_*M(1,4)$}\label{sec:tmfM14}

In this section we describe a complete computation of 
the MASS
$$  \Ext^{s,t}_{A(2)_*}(H(1,4))
\Rightarrow \pi_{t-s}(\tmf \wedge M(1,4)). $$

The spectral sequence is displayed in four pages in
Figures~\ref{fig:tmfASS12} and \ref{fig:tmfASS34}.  The entire spectral
sequence is $v_2^{32}$-periodic.

\begin{figure}
\includegraphics[width=0.462\textwidth]{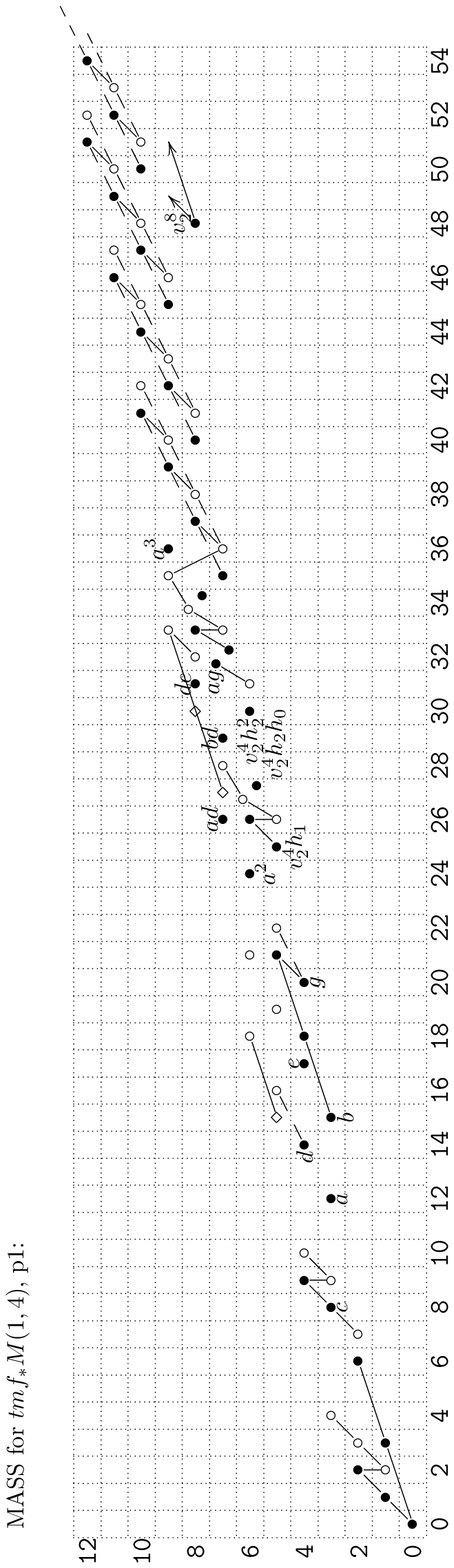} 
\hfill
\includegraphics[width=0.498\textwidth]{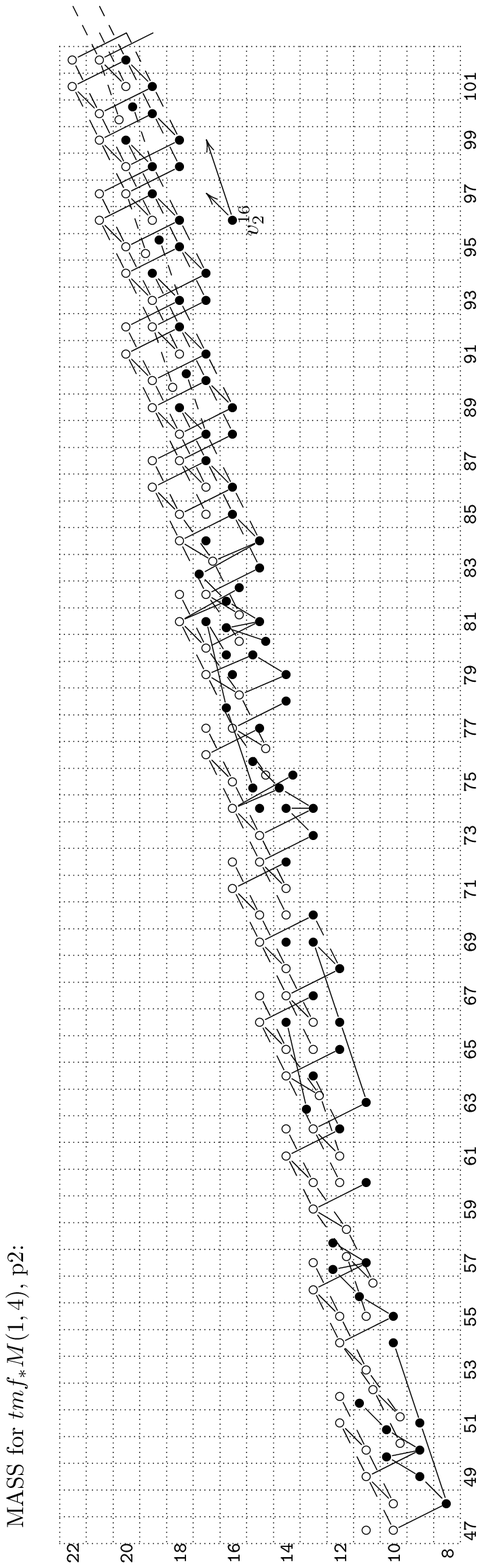}
\caption{}\label{fig:tmfASS12}
\end{figure}

\begin{figure}
\includegraphics[width=0.456\textwidth]{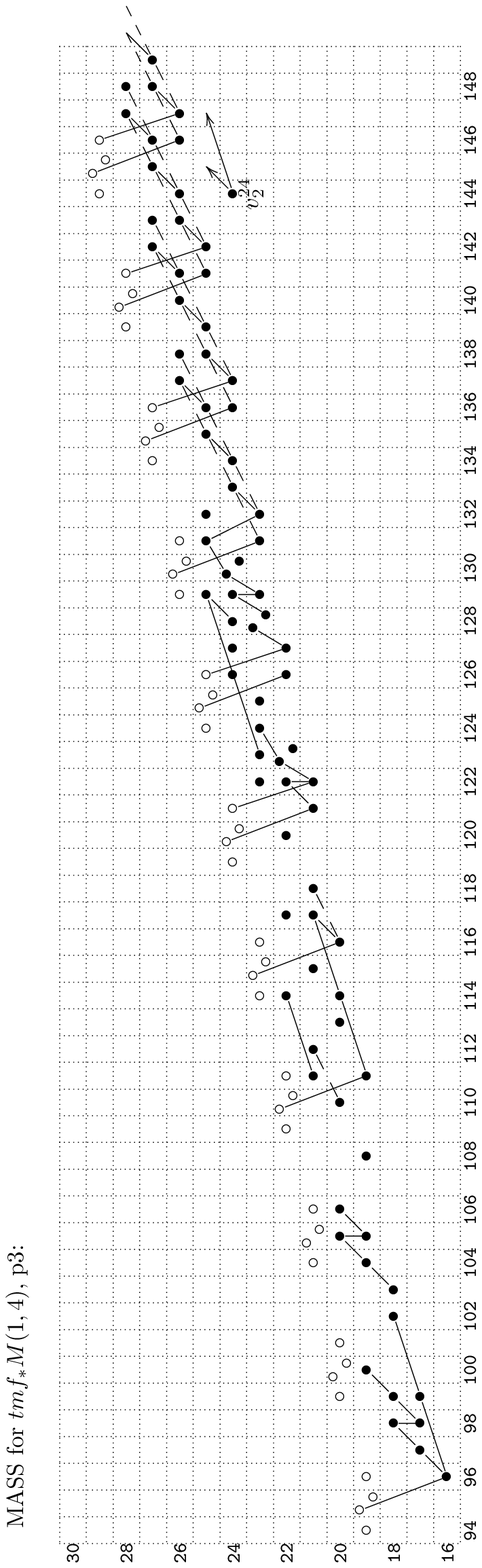}
\hfill
\includegraphics[width=0.504\textwidth]{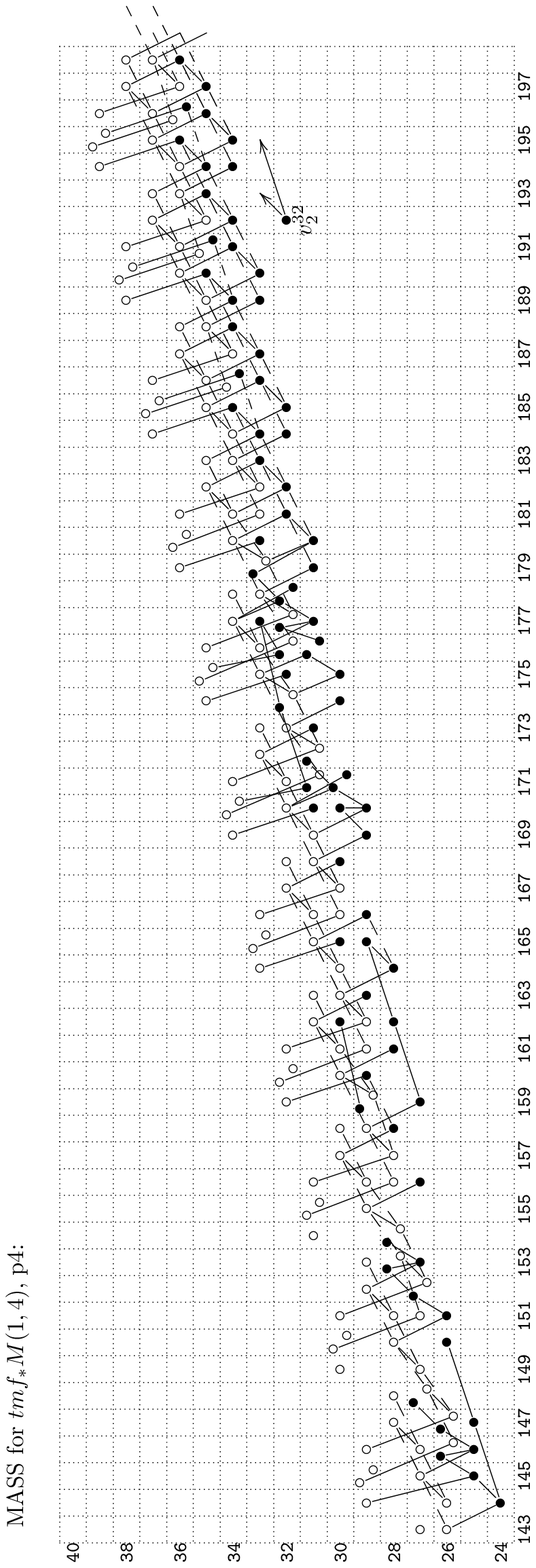}
\caption{}\label{fig:tmfASS34}
\end{figure}

We first explain what is happening in these charts.  Then we explain the
methodology used to produce these differentials.

\subsection*{Page 1: dimensions 0--48}

The truncated wedge beginning in $t-s = 35$ is infinite, and propagated by
$g$-multiplication.  The entire chart is periodic under
$v_2^8$-multiplication.  
Classes born on the $0$-cell of $M(1)$ are denoted
with a $\bullet$, and classes born on the $1$-cell of $M(1)$ are denoted
with a $\circ$.  Although multiplication by $c_4 = v_1^4$ is faithful in 
$\Ext_{A(2)_*}(\FF_2)$, it is not faithful in $\Ext_{A(2)_*}(M(1))$.  We
therefore get some classes coming from the $9$-cell of $M(1,4)$, which we
denote with a $\diamond$.

There are only two possible Adams differentials through $t-s = 47$, and
only one of them actually occurs.  This differential is indicated on the
chart.

\subsection*{Page 2: dimensions 48--96}

We move to the region between the occurrence of $v_2^8$
and $v_2^{16}$.  There are numerous $d_2$ differentials in this range, 
displayed in the
chart.  In this chart, the classes propagated by $v_2^8$ are denoted with
a $\bullet$, and the classes coming from the truncated wedge starting in
$t-s = 35$ are denoted with a $\circ$.
Note that beginning in $t-s = 91$, we just have the following pattern.
\sseqxstart=89
\sseqystart=18
\begin{sseq}{13}{3}

\ssmoveto{89}{18}
\ssdrop{\bullet}
\ssmove 1 0
\ssdrop{\circ}
\ssdrop{\bullet}
\ssmove 1 0
\ssdrop{\circ}
\ssmove 3 1

\ssdrop{\bullet}
\ssmove 1 0
\ssdrop{\circ}
\ssdrop{\bullet}
\ssmove 1 0
\ssdrop{\circ}
\ssmove 3 1

\ssdrop{\bullet}
\ssmove 1 0
\ssdrop{\circ}
\ssdrop{\bullet}
\ssmove 1 0
\ssdrop{\circ}
\end{sseq}

\subsection*{Page 3: dimensions 96--144}

We now move up to the region between $v_2^{16}$ and $v_2^{24}$.  We
propagate only the $h_{2,1}$-periodic pattern from the previous page 
(denoted with $\circ$); everything else is either the source or
target of a $d_2$.  We denote the elements propagated by $v_2^{16}$
multiplication with a $\bullet$.

\subsection*{Page 4: dimensions 144--192}

We now introduce the differentials supported by $v_2^{24}$ and its
multiples.  We see that eventually we get a small gap in homotopy between
the $180$ stem and the $192$ stem.  Then the pattern repeats with
$v_2^{32}$-periodicity.

\subsection*{Methodology}

In \cite{HopkinsMahowald}, the structure of the Adams spectral sequence
$$ \Ext^{*,*}_{A(2)_*}(\FF_2) \Rightarrow \pi_*\tmf_2 $$
is completely determined.  The Adams spectral sequence for $\tmf_*M(1,4)$
is a module over the Adams spectral sequence for $\tmf_*$, and all of the
differentials for $\tmf_*M(1,4)$ were deduced from this structure.  These
computations were double-checked against the Atiyah-Hirzebruch spectral
sequence
$$ H^*(M(1,4), \tmf_*) \Rightarrow \tmf_*M(1,4) $$
using the known values of $\tmf_*$.  As a further consistency check, a
combination of Gross-Hopkins duality 
\cite{GrossHopkins}
and Mahowald-Rezk 
\cite{MahowaldRezk}
duality shows that
$\tmf_*M(1,4)$ is, up to a shift, Pontryagin self-dual, and this is
consistent with our computations.

\section{$d_2(v_2^8)$ and $d_3(v_2^{16})$}\label{sec:d2d3}

In this section we will lift the differentials $d_2(v_2^{8})$ and
$d_3(v_2^{16})$ from the MASS for
$\tmf_*M(1,4)$ to the MASS for $\pi_*(M(1,4) \wedge DM(1,4))$.  We will observe
that both $d_2(v_2^8)$ and $d_3(v_2^{16})$ are central, and hence, using
the fact that the MASS for $\pi_*(M(1,4)\wedge DM(1,4))$ is a spectral
sequence of algebras, we will deduce that $d_r(v_2^{32})$ is zero for $r <
4$.

\begin{lem}\label{lem:d2v28}
In the MASS for $\pi_*(M(1,4) \wedge DM(1,4))$, there is a
differential 
$$ d_2 (v_2^8) = \td{e_0 r}, $$
where $\td{e_0 r}$ is the image of the element $e_0 r$ under the map
$$ \Ext^{10,57}_{A_*}(\FF_2) \rightarrow \Ext^{10,57}_{A_*}(H(1,4) \otimes
DH(1,4)). $$
\end{lem}

\begin{proof}
By Proposition~\ref{prop:DavisMahowald} it suffices to establish that
$d_2(v_2^8) = \td{e_0 r}$ in the MASS
$$ \Ext^{s,t}_{A_*}(H(1,4) \otimes DH(1)) \Rightarrow \pi_{t-s} (M(1,4)
\wedge DM(1)). $$
The differential $d_2(v_2^8)$ in the Adams spectral sequence for $\tmf$ maps to a
differential $d_2(v_2^8) = \tdd{e_0 r}$ 
under the map of (M)ASSs
$$
\xymatrix{
\Ext^{s,t}_{A(2)_*}(\FF_2) \ar@{=>}[r] \ar[d] &
\pi_{t-s}\tmf \ar[d] \\
\Ext^{s,t}_{A(2)_*}(H(1,4) \otimes DH(1)) \ar@{=>}[r] &
\tmf_{t-s}(M(1,4) \wedge DM(1))
}
$$
where $\tdd{e_0 r}$ is the image of $e_0 r$ under the composite
$$ \Ext^{10, 57}_{A_*}(\FF_2) \rightarrow \Ext^{10,57}_{A_*}(H(1,4)\otimes
DH(1)) \rightarrow \Ext^{10,57}_{A(2)_*}(H(1,4)\otimes DH(1)). $$
We wish to lift the differential $d_2(v_2^8) = \tdd{e_0 r}$ to $d_2(v_2^8)
= \td{e_0 r}$
using the map of MASSs:
$$
\xymatrix{
\Ext^{s,t}_{A_*}(H(1,4) \otimes DH(1)) \ar@{=>}[r] \ar[d] &
\pi_{t-s}(M(1,4) \wedge DM(1)) \ar[d] \\
\Ext^{s,t}_{A(2)_*}(H(1,4) \otimes DH(1)) \ar@{=>}[r] &
\tmf_{t-s}(M(1,4) \wedge DM(1))
}
$$
However, using 
$$ \Ext^{s,t}_{A(2)_*}(H(1,4)) = 
\begin{cases}
e_0 r[0] & (t-s,s) = (47,10), \\
e_0 r[1] & (t-s,s) = (48,10)
\end{cases}
$$
and
$$ \Ext^{s,t}_{A_*}(\FF_2) = 
\begin{cases}
e_0 r & (t-s,s) = (47,10), \\
0 & (t-s,s) = (46,10), (48,10), (55,5), (56,5), (57,5)
\end{cases}
$$
we may deduce that the map
$$ \Ext^{10,57}_{A_*}(H(1,4)\otimes DH(1)) \rightarrow
\Ext^{10,57}_{A(2)_*}(H(1,4)\otimes DH(1))
$$
is an isomorphism.  This suffices to show that the differential
$d_2(v_2^8)$ lifts as
desired.
\end{proof}

Since $d_2(v_2^8)$ is central, Proposition~\ref{prop:MASSmult} gives 
the following corollary.

\begin{cor}\label{cor:d2v216}
In the MASS for $\pi_*(M(1,4) \wedge DM(1,4))$, we have $d_2 (v_2^{16}) = 0$.
\end{cor}

We now investigate $d_3(v_2^{16})$.

\begin{lem}
In the MASS for $\pi_*(M(1,4)\wedge DM(1,4))$, the element $d_3(v_2^{16})$
is in the image of
$$ \Ext^{19,114}_{A_*}(\FF_2) \rightarrow \Ext^{19,114}_{A_*}(H(1,4)\otimes
DH(1,4)). $$
In particular, $d_3(v_2^{16})$ is central.
\end{lem}

\begin{proof}
By Proposition~\ref{prop:DavisMahowald} it suffices to establish that
in the MASS for $\pi_*(M(1,4)\wedge DM(1))$, 
the element $y = d_3(v_2^{16})$  
is in the image of the map
$$ \Ext^{19,114}_{A_*}(\FF_2) \rightarrow \Ext^{19,114}_{A_*}(H(1,4)\otimes
DH(1)). $$
The differential $d_3(v_2^{16})$ in the ASS for $\tmf$ maps to a
differential $d_3(v_2^{16}) = z$ under the map of (M)ASSs
$$
\xymatrix{
\Ext^{s,t}_{A(2)_*}(\FF_2) \ar@{=>}[r] \ar[d] &
\pi_{t-s}\tmf \ar[d] \\
\Ext^{s,t}_{A(2)_*}(H(1,4) \otimes DH(1)) \ar@{=>}[r] &
\tmf_{t-s}(M(1,4) \wedge DM(1))
}
$$
where $z$ is in the image of
$$ \Ext^{19,114}_{A(2)_*}(\FF_2) \rightarrow
\Ext^{19,114}_{A(2)_*}(H(1,4)\otimes DH(1)). $$
Using the map of spectral sequences
$$
\xymatrix{
\Ext^{s,t}_{A_*}(H(1,4) \otimes DH(1)) \ar@{=>}[r] \ar[d] &
\pi_{t-s}(M(1,4) \wedge DM(1)) \ar[d] \\
\Ext^{s,t}_{A(2)_*}(H(1,4) \otimes DH(1)) \ar@{=>}[r] &
\tmf_{t-s}(M(1,4) \wedge DM(1))
}
$$
we see that $y$ maps to $z$.  Therefore $z$ detects $y$ in the algebraic
$\tmf$-resolution for $\Ext^{*,*}_{A_*}(H(1,4)\otimes DH(1))$.  Since the
algebraic $\tmf$-resolution is functorial, we deduce that $y$ is in the image
of the map 
$$ \Ext^{19,114}_{A_*}(\FF_2) \xrightarrow{i_*} 
\Ext^{19,114}_{A_*}(H(1,4)\otimes
DH(1)) $$
modulo higher terms of the algebraic $\tmf$-resolution: that is to say,
there exists an element
$$ x \in \Ext^{19,114}_{A_*}(\FF_2) $$
such that $y - i_*(x)$ is detected in a higher filtration of the algebraic
$\tmf$-resolution.

We are left with showing that $w = y - i_*(x) = 0$.  Suppose not.  Using our
vanishing lines from Section~\ref{sec:reduction} and our $\Ext_{A(2)_*}$
computations from Section~\ref{sec:Charts}, we deduce that $w$ 
is detected in the algebraic $\tmf$-resolution by an element
$$ \br{w} \in \Ext^{19,114}_{A(2)_*}(M_2(1) \otimes H(1,4) \otimes DH(1)[-1]) $$
and the image of $\br{w}$ under the 
map
$$ \Ext^{19,114}_{A(2)_*}(M_2(1) \otimes H(1,4) \otimes DH(1)[-1])
\xrightarrow{1 \otimes p_* \otimes 1}
\Ext^{19,114}_{A(2)_*}(M_2(1) \otimes \Sigma^{12}H(1) \otimes DH(1)[-4]) $$
is non-trivial, where $p_*$ is the projection
$$ H(1,4) \rightarrow \Sigma^{12} H(1)[-3] $$
in the derived category of $A_*$-comodules induced by the projection
$$ p: M(1,4) \rightarrow \Sigma^9 M(1). $$
We deduce that in the MASS for $M(1) \wedge DM(1)$ there is a
differential
$$ d_3((p_* \otimes 1)(v_2^{16})) = (p_* \otimes 1) (w). $$
We will verify the following claim:

\begin{claim}\label{eq:claim}
The element $(p_* \otimes 1)(w)$ is non-trivial in the
$E_3$-page of the MASS for $M(1) \wedge DM(1)$.
\end{claim}

Assuming Claim~\ref{eq:claim}, 
we deduce that $d_3((p_* \otimes 1)(v_2^{16}))$ is
non-trivial.  However, the image of $v_2^{16}$ under the map
$$ \Ext^{16, 112}_{A_*}(H(1,4)\otimes DH(1)) \xrightarrow{p_* \otimes 1}  
\Ext^{16, 112}_{A_*}(\Sigma^{12} H(1)\otimes DH(1)[-3]) $$
may be computed using the May spectral sequence.  
In the May spectral sequence, the element $v_2^{16}$ is detected by
$b_{3,0}^8$.
Applying Nakamura's formula
\cite{Nakamura} to the May spectral sequence 
differential $d_8(b_{3,0}^4) = h_5 b_{2,0}^4$ in the proof of
Proposition~\ref{prop:v_2^8} gives
$$ d_{16}(b_{3,0}^8) = h_6 b_{2,0}^8 $$
from which it follows that 
$$ (p_* \otimes 1)(v_2^{16}) = h_6 b_{2,0}^6. $$
The element $b_{2,0}^6$ detects the cube of the Adams map:
$$ v_1^{12} = (v_1^4)^3 \in \pi_{24}(M(1) \wedge DM(1)). $$
Since this homotopy element has order $2$, the Adams differential
$$ d_2(h_6) = h_0 h_5^2 $$
implies that the element $h_6 b_{2,0}^6$
detects the Toda bracket of the composite
$$ 
S^{86} \xrightarrow{\theta_5} 
S^{24} \xrightarrow{2} S^{24} \xrightarrow{v_1^{12}} M(1) \wedge DM(1).
$$
In particular, $h_6 b_{2,0}^6$ is a permanent cycle in the MASS for $M(1)
\wedge DM(1)$, which contradicts the existence of a non-trivial
differential $d_3((p_* \otimes 1)(v_2^{16}))$.  Thus the
assumption that $w \ne 0$ gives rise to a contradiction, and we conclude
that $w = 0$, as desired.

We are left with verifying Claim~\ref{eq:claim}.  We will verify this claim
by establishing:
\begin{enumerate}
\item The element 
$$ (1 \otimes p_* \otimes 1)(\br{w}) \in \Ext_{A(2)_*}^{19,114}(M_2(1)\otimes \Sigma^{12} H(1)
\otimes DH(1)[-4]) $$ 
is not the target of a differential in the algebraic $\tmf$-resolution for
$\Ext_{A_*}^{*,*}(H(1) \otimes DH(1))$.
\item The element 
$$ (p_* \otimes 1)(w) \in \Ext_{A_*}^{19,114}(\Sigma^{12} H(1) \otimes DH(1)[-3]) $$ 
is not the target of a $d_2$
differential in the MASS for $M(1) \wedge DM(1)$.
\end{enumerate}

Item (1) above is verified by observing that
$$ \Ext^{s,t}_{A(2)_*}(\FF_2) = 0 \qquad (t-s,s) = (86, 15), (87,15),
(88,15) $$
and so there are no possible contributions to 
$$ \Ext^{87,15}_{A(2)_*}(H(1)\otimes DH(1)), $$
and this is the only possible source for a differential in the algebraic
$\tmf$-resolution.

We now verify (2).  The $A_*$-comodule $H(1) \otimes DH(1)$ has the following
diagram of generators.
\begin{equation}\label{eq:H1DH1}
\xymatrix@C-2em@R-2em{
1 & \circ \ar@{-}[d] \ar@{-} `[r] `[rrdd] [rrdd] & \\
0 & \bullet  && \triangle \ar@{-}[d] \\
-1 & && \square
}
\end{equation}
Here the straight lines encode the action of $Sq^1_*$ 
and the curved line denotes a
$Sq^2_*$.  Using Bruner's computer generated $\Ext_{A_*}(\FF_2)$ charts
\cite{Bruner}, 
we compute the
vicinity of $(p_* \otimes 1)(w)$ in $\Ext^{*,*}_{A_*}(H(1)\otimes DH(1))$
in Table~\ref{tab:p*w}.

\begin{table}[!h]
\begin{tabular}{c|c|c}
$s \backslash t-s$  & 86 & 87
\\
\hline
16 
& 
$\begin{array}{c}
\circ \circ \\
\triangle \\
(p_* \otimes 1)(w) \ssb \ssb \phantom{(p_* \otimes 1)(w)}
\end{array}$ 
&
$\ast$
\\
\hline
15
&
$\ast$
&
$\ast$
\\
\hline
14
&
$\ast$
&
$\begin{array}{c}
\phantom{b_{87}} \circ b_{87} \\
\phantom{a_{87}} \square a_{87}
\end{array}$
\end{tabular}
\caption{$\Ext^{s,t}_{A_*}(H(1)\otimes
DH(1))$ near $(p_* \otimes 1)(w)$}\label{tab:p*w}
\end{table}

In this table, entries marked with $\ast$ are not computed, otherwise,
elements are denoted by the generator (as in (\ref{eq:H1DH1})) that
supports it.  The only possible sources for a non-trivial $d_2$ are
$a_{87}$ and $b_{87}$.

The element $a_{87}$ is the image of an element
$$ a_{88} \in \Ext^{14, 102}_{A_*}(\FF_2) $$
under the inclusion of the bottom generator
$$ \Sigma^{-1} \FF_2 \rightarrow H(1) \otimes DH(1). $$
Since $\Ext^{16, 103}_{A_*}(\FF_2) = 0$, we deduce that $d_2(a_{88}) = 0$
in the ASS for $\pi_* S$.  The map of MASSs induced from the inclusion of
the bottom cell of $M(1) \wedge DM(1)$ gives $d_2(a_{87}) = 0$.

We now turn our attention to $b_{87}$.  Table~\ref{tab:p*w2} shows the
portion of
$\Ext_{A_*}(H(1)\otimes DH(1))$ mapped to the vicinity of
Table~\ref{tab:p*w} under $h_2$-multiplication.

\begin{table}[!h]
\begin{tabular}{c|c|c}
$s \backslash t-s$  & 83 & 84
\\
\hline
15 
& 
$\begin{array}{c}
c_{83} \bullet \phantom{c_{83}} \\
c'_{83} \square \square c''_{83}
\end{array}$ 
&
$\ast$
\\
\hline
14
&
$\ast$
&
$\ast$
\\
\hline
13
&
$\ast$
&
$\begin{array}{c}
\phantom{b_{84}} \circ b_{84} \\
\square
\end{array}$
\end{tabular}
\caption{$\Ext^{s,t}_{A_*}(H(1)\otimes
DH(1))$ near $h_2^{-1} b_{87}$}\label{tab:p*w2}
\end{table}

Using the $h_2$ multiplicative structure in Bruner's tables \cite{Bruner}, 
we deduce that
\begin{align*}
h_2 b_{84} & = b_{87}, \\
h_2 c_{83} & = 0, \\
h_2 c'_{83} & = 0, \\
h_2 c''_{83} & = 0.
\end{align*}
Since $h_2$ is a permanent cycle in the ASS for the sphere, we have
$$ d_2(b_{87}) = d_2(h_2 b_{84}) = h_2 d_2(b_{84}) = 0. $$
This completes our proof of Claim~\ref{eq:claim}.
\end{proof}

Proposition~\ref{prop:MASSmult} gives 
the following corollary.

\begin{cor}\label{cor:d3v232}
In the MASS for $\pi_*(M(1,4) \wedge DM(1,4))$, we have $d_3 (v_2^{32}) = 0$.
\end{cor}

\section{Calculation of an Adams differential}\label{sec:diff}

The image of the element 
$\bar{\kappa} \in \pi_{20}(S)_{2}$ in $\pi_{20}(M(1,4) \wedge DM(1,4))$
gives rise to a self-map
$$ \td{\kappa} : M(1,4) \rightarrow M(1,4). $$
The element $g \in \Ext^{4,24}_{A_*}(\FF_2)$ which detects $\bar{\kappa}$
maps to a permanent cycle $\td{g} \in \Ext_{A_*}(H(1,4) \otimes DH(1,4))$
which
detects $\td{\kappa} \in \pi_{20}(M(1,4) \wedge DM(1,4))$ in the MASS.
The purpose of this section is to prove the following theorem.

\begin{thm}\label{thm:diff}
$\quad$
\begin{enumerate}
\item
The element $v_2^{20}h_1 \in \Ext^{21,142}_{A(2)_*}(H(1,4))$ lifts to an
element 
$$ \td{v_2^{20}h_1} \in \Ext^{21, 142}_{A_*}(H(1,4) \otimes DH(1,4)).$$
\item
There is a differential 
$$ d_3(\td{v_2^{20}h_1}) = \td{g}^6 + R$$
in the MASS for $M(1,4) \wedge DM(1,4)$, where $R$ is an element of
filtration greater than $0$ in the algebraic $\tmf$-resolution.
\end{enumerate}
\end{thm}

\begin{proof}
% By Corollary~\ref{cor:DavisMahowald}, to prove statement~(1) it suffices
% to show that the image of $\bar{\kappa}$ in $\pi_*(M(1,4))$ has order $2$.
% The homotopy element $2\bar{\kappa}$ has Adams filtration greater than
% $4$.
% Using low dimensional knowledge of $\Ext_{A_*}(\FF_2)$, it is easy to see
% that 
% \begin{equation}\label{eq:2kappabar}
% \Ext^{s, 20+s}_{A_*}(H(1,4)) = 0
% \end{equation}
% for $s > 4$.  Thus the image of
% $\bar{\kappa}$ in $\pi_*(M(1,4))$ has order $2$.
% 
% By Lemma~\ref{lem:algDavisMahowald}, to prove statement (2) it suffices
% to prove that the element $h_0g \in \Ext^{5,25}_{A_*}(H(1,4))$ is null,
% which again follows from (\ref{eq:2kappabar}).
% 
Table~\ref{tab:v_2^20h_1} displays a small portion of 
the $E_1$-page of the algebraic $\tmf$-resolution for $\Ext_{A_*}(H(1,4)
\wedge DH(1))$.

\begin{table}[!h]
\begin{tabular}{c|c|c}
$s \backslash t-s$  & 120 & 121
\\
\hline
24 
& 
$\begin{array}{c}
\phantom{g^6} \ssb \ssb \ssb g^6 \\
\ssc \ssc
\end{array}$ 
&
$\begin{array}{c}
\ssb \ssb \\
\ssc
\end{array}$
\\
\hline
23
&
$\begin{array}{c}
a_{120} \ssb \ssb b_{120} \\
\phantom{x_{120}} \ssc x_{120} \\
\phantom{y_{120}} \sscc y_{120}
\end{array}$
&
$\begin{array}{c}
\ssb \ssb \\
\ssc \ssc \\
\sscc \sscc
\end{array}$
\\
\hline
22
&
$\begin{array}{c}
\ssb \\
\phantom{z_{120}} \sscc z_{120}
\end{array}$
&
$\begin{array}{c}
\sscc \\
*
\end{array}$
\\
\hline
21
&
$\begin{array}{c}
\ssb \\
\ssc \\
\sscc \sscc \sscc \sscc \sscc \\
\ssccc \ssccc \ssccc \ssccc \\
*
\end{array}$
&
$\begin{array}{c}
v_2^{20}h_1 \ssb \ssb \phantom{v_2^{20}h_1}\\
\sscc \sscc \sscc \sscc \\
\ssccc \ssccc \\
*
\end{array}$
\end{tabular}
\caption{The algebraic $\tmf$-resolution for $\Ext_{A_*}(H(1,4) \otimes
DH(1))$ near $v_2^{20}h_1$}\label{tab:v_2^20h_1}
\end{table}

In this and all future tables depicting the algebraic $\tmf$-resolution, we
have the following key:
\begin{align*}
\bullet & = \text{generator of $\Ext_{A(2)_*}(H(1,4) \otimes DH(1))$,} \\
\circ & = \text{generator of $\Ext_{A(2)_*}(M_2(1)[-1] \otimes H(1,4) \otimes
DH(1))$,} \\
\scriptstyle{\odot} & = 
\text{generator of $\Ext_{A(2)_*}(M_2(1)^{\otimes 2}[-2]
\otimes H(1,4)\otimes DH(1))$,} \\
\scriptstyle{\circledcirc} & = 
\text{generator of $\Ext_{A(2)_*}(M_2(1)^{\otimes 3}[-3]
\otimes H(1,4) \otimes DH(1))$,} \\
* & = \text{potential contribution from} \\ 
& \qquad \qquad \Ext_{A(2)_*}(M_2(j_1) \otimes
\cdots \otimes M_2(j_n)[-n] \otimes H(1,4) \otimes DH(1)) \\
& \qquad \text{where either for some $i$, $j_i > 1$, or $n > 3$.}
\end{align*}
We shall refer to all differentials in the algebraic $\tmf$-resolution as
$d_1$ differentials.  Differentials in the MASS
$$ E_2^{s,t} =  \Ext^{s,t}_{A_*}(H(1,4) \otimes DH(1)) \Rightarrow \pi_{t-s}(M(1,4)
\wedge DM(1)) $$
will be referred to by $d_r$ for $r \ge 2$.

In order to prove (1), we must show that the element $v_2^{20}h_1$ in
Table~\ref{tab:v_2^20h_1} does not
support a non-trivial $d_1$.
There is one possible target $z_{120}$ in $(t-s,s)
= (120, 22)$, but we will argue shortly that this possibility cannot occur.
Assuming for the moment that $d_1(v_2^{20}h_1) = 0$,
we would conclude that $v_2^{20}h_1$ lifts to an element
$$ \td{v_2^{20}h_1} \in \Ext^{21, 142}_{A_*}(H(1,4)\otimes DH(1,4)). $$
The composite
$$ H(1,4) \wedge DH(1,4) \rightarrow H(1,4) \rightarrow \tmf \wedge H(1,4)
$$
induces a map of MASSs:
$$
\xymatrix{
\Ext^{s,t}_{A_*}(H(1,4) \otimes DH(1,4)) \ar@{=>}[r] \ar[d] 
& \pi_{t-s}(M(1,4) \wedge DM(1,4)) \ar[d] 
\\
\Ext^{s,t}_{A(2)_*}(H(1,4)) \ar@{=>}[r] 
& \pi_{t-s}(\tmf \wedge M(1,4))
}
$$
In the MASS for $\tmf \wedge M(1,4)$, there is a differential
$$ d_3(v_2^{20}h_1) = g^6. $$
In order to prove (2), we need to lift this differential to the MASS for $M(1,4) \wedge DM(1,4)$.
By Proposition~\ref{prop:DavisMahowald}, it suffices to lift this
differential to the MASS for $M(1,4) \wedge DM(1)$:
$$ \Ext^{s,t}_{A_*}(H(1,4) \otimes DH(1)) \Rightarrow
\pi_{t-s}(M(1,4)\wedge DM(1)). $$
The obstruction to lifting this differential is that $\td{v_2^{20}h_1}$
could support a $d_2$ in the MASS for $M(1,4) \wedge DM(1)$.
In fact, Table~\ref{tab:v_2^20h_1} demonstrates that there are four possible
targets for such a $d_2$ in $(t-s,s) = (120, 23)$: these are labeled
$a_{120}$, $b_{120}$, $x_{120}$, $y_{120}$.

We now argue (1) and (2) by showing that the element
$v_2^{20}h_1$ in Table~\ref{tab:v_2^20h_1} cannot support a non-trivial
$d_1$ or $d_2$.  We will need Tables~\ref{tab:gv_2^20h_1} and
\ref{tab:gv_2^4h_1}, which depict the $\tmf$-resolution in the vicinities of 
$gv_2^{20}h_1$ and $gv_2^4h_1$, respectively.

\begin{table}[!h]
\begin{tabular}{c|c|c}
$s \backslash t-s$ & 140 & 141
\\
\hline
27 
&
$\begin{array}{c}
ga_{120} \ssb \ssb gb_{120} \\
\phantom{gx_{120}} \ssc gx_{120} \\
\phantom{gy_{120}} \sscc gy_{120}
\end{array}$
&
$\begin{array}{c}
\ssb \\
\ssc \ssc \\
\sscc \sscc 
\end{array}$
\\
\hline
26
&
$\begin{array}{c}
\ssb \ssb \\
\phantom{gz_{120}} \sscc gz_{120}
\end{array}$
&
$\begin{array}{c}
\phantom{x_{141}} \ssb x_{141} \\
\sscc \\
*
\end{array}$
\\
\hline
25
&
$\begin{array}{c}
\ssb \\
\ssc \ssc \ssc \\
\sscc \sscc \sscc \\
\ssccc \ssccc \ssccc \ssccc \\
*
\end{array}$
&
$\begin{array}{c}
gv_2^{20}h_1 \ssb \ssb \phantom{gv_2^{20}h_1} \\
\ssc \ssc \\
\sscc \sscc \\
\ssccc \ssccc \\
*
\end{array}$
\end{tabular}
\caption{The algebraic $\tmf$-resolution for $\Ext_{A_*}(H(1,4)\otimes
DH(1))$ near $gv_2^{20}h_1$}\label{tab:gv_2^20h_1}
\end{table}

\begin{table}[!h]
\begin{tabular}{c|c|c}
$s \backslash t-s$ & 44 & 45
\\
\hline
11 
&
&
$\ssb$
\\
\hline
10
&
$\ssb \ssb$
&
$\phantom{v_2^{-16}x_{141}} \ssb v_2^{-16}x_{141}$
\\
\hline
9
&
$\begin{array}{c}
\ssb \\
\ssc \ssc
\end{array}$
&
$\begin{array}{c}
gv_2^4h_1 \ssb \ssb \phantom{gv_2^4h_1} \\
\ssc \ssc \\
*
\end{array}$
\end{tabular}
\caption{The algebraic $\tmf$-resolution for $\Ext_{A_*}(H(1,4)\otimes
DH(1))$ near $gv_2^{4}h_1$}\label{tab:gv_2^4h_1}
\end{table}

Write $d_1(v_2^{20}h_1) = c\cdot z_{120}$ for $c \in \FF_2$.  Then have
$$ d_1(gv_2^{20}h_1) = c \cdot gz_{120}. $$
Table~\ref{tab:gv_2^4h_1} shows that $d_1(g v_2^4 h_1) = 0$.  Multiplying
by the $d_2$-cycle $v_2^{16}$ of Corollary~\ref{cor:d2v216}, we deduce that we
must have $d_1(gv_2^{20}h_1) = 0$.  
Thus $c$ equals $0$, and we have proven (1).

Write
$$ d_2(v_2^{20}h_1) = c_1 a_{120} + c_2 b_{120} + c_3 x_{120} + c_4 y_{120} $$
for $c_i \in \FF_2$.  The image of $v_2^{20}h_1$ in $\Ext_{A(2)_*}(H(1,4))$
is a $d_2$-cycle in the MASS
$$ \Ext^{s,t}_{A(2)_*}(H(1,4)) \rightarrow \pi_{t-s}(\tmf \wedge M(1,4)). $$
We therefore deduce that $c_1 = c_2 = 0$.  We wish to show that
$d_2(v_2^{20}h_1) = 0$, i.e. that it is contained in the image of $d_1$.
We have
$$ d_2(gv_2^{20}h_1) = c_3 gx_{120} + c_4 g y_{120}. $$
Examining Table~\ref{tab:gv_2^4h_1}, we see that $d_2(gv_2^4h_1) = 0$.
Since $v_2^{16}$ is a $d_2$-cycle, we deduce that $d_2(gv_2^{20}h_1) = 0$.
This means that
$$ c_3 gx_{120} + c_4 gx_{120} $$
is in the target of a $d_1$.
With the exception of the element $x_{141}$, all of the generators in
$(t-s,s) = (141, 26)$ are $g$-periodic.  Thus we have
$$ d_1(E_1^{26,167}) = g\cdot d_1(E_1^{22, 143}) + \FF_2\{ d_1(x_{141})\} $$
where $E_1^{s,t}$ is the $E_1$-term of the algebraic $\tmf$-resolution for
$\Ext^{*,*}_{A_*}(H(1,4)\otimes DH(1))$.  However, we see from
Table~\ref{tab:gv_2^4h_1} that $d_1(v_2^{-16}x_{141}) = 0$, so it follows
that $d_1(x_{141}) = 0$.  We may therefore deduce the vanishing of
$d_2(v_2^{20}h_1)$ from the vanishing of $d_2(gv_2^{20}h_1)$.  We have
proven (2).
\end{proof}

\section{Proof of the main theorem}\label{sec:mainthm}

By Proposition~\ref{prop:v_2^8} and Lemma~\ref{lem:resolutionmult}, 
the element 
$$ v_2^{32} \in \Ext^{32,224}_{A(2)_*}(H(1,4) 
\otimes DH(1)) $$
is a permanent cycle in the
algebraic $\tmf$-resolution, and it detects an element
$$ v_2^{32} \in \Ext^{32,224}_{A_*}(H(1,4) \otimes DH(1)). $$
By Corollary~\ref{cor:d3v232}, the element
$v_2^{32}$ persists to the $E_4$-page of the MASS for $M(1,4) \wedge
DM(1)$.  By Proposition~\ref{prop:DavisMahowald}, our main theorem 
(Theorem~\ref{thm:mainthm}) is a consequence of the following lemma.

\begin{lem}
The element
$$ v_2^{32} \in \Ext^{32,224}_{A_*}(H(1,4) 
\otimes DH(1)) $$
cannot support a non-trivial $d_r$ in the MASS for $M(1,4) \wedge DM(1)$
for $r \ge 4$.
\end{lem}

\begin{proof}
We shall make use of the following tables.  Table~\ref{tab:v_2^32a} depicts
the algebraic $\tmf$-resolution for $\Ext^{*,*}_{A_*}(H(1,4) \otimes
DH(1))$ in the region where all possible targets of
$d_r(v_2^{32})$ can lie, for $r \ge 4$.  Note that there are no non-zero
elements in the algebraic $\tmf$-resolution that can contribute to
$\Ext^{s,191+s}_{A_*}(H(1,4) \otimes DH(1))$ for $s > 40$.
Table~\ref{tab:v_2^32b} depicts a region of the algebraic $\tmf$-resolution
which maps to the region of Table~\ref{tab:v_2^32a} under
$g^6$-multiplication.  The notation in these tables is explained in
Section~\ref{sec:diff}.  The subgroups labeled $G_{191}$ and $G_{71}$ are
the subgroups generated by the contributions in the algebraic
$\tmf$-resolution labeled with a $*$.

\begin{table}[!h]
\begin{tabular}{c|c|c}
$s \backslash t-s$ & 190 & 191
\\
\hline
40
&
$\ssb$
&
$\ssb \ssb$
\\
\hline
39
&
$\ssb$
&
$\ssb$
\\
\hline
38
&
$\begin{array}{c}
\ssb \ssb \ssb \\
g^6b_{70} \ssc \ssc g^6c_{70}
\end{array}$
&
$\begin{array}{c}
\ssb \ssb \\
g^6f_{71} \ssc \phantom{g^6f_{71}}
\end{array}$
\\
\hline
37
&
$\begin{array}{c}
\ssb \ssb \\
\ssc  \\
g^6a_{70} \sscc \phantom{g^6a_{70}}
\end{array}$
&
$\begin{array}{c}
v_2^8 k_{143} \ssb \ssb v_2^8 l_{143} \\
g^6d_{71} \ssc \ssc g^6e_{71} \\
g^6b_{71} \sscc \sscc g^6 c_{71}
\end{array}$
\\
\hline
36
&
$\begin{array}{c}
\ssb \ssb \\
\sscc
\end{array}$
&
$\begin{array}{c}
\ssb \\
g^6a_{71} \sscc \phantom{g^6a_{71}} \\
G_{191} * \phantom{G_{191}}
\end{array}$
\end{tabular}
\caption{The algebraic $\tmf$-resolution for $\Ext^{*,*}_{A_*}(H(1,4)
\otimes DH(1))$ in the vicinity of $(t-s,s) = (191, 36)$}\label{tab:v_2^32a}
\end{table}

\begin{table}[!h]
\begin{tabular}{c|c|c}
$s \backslash t-s$ & 70 & 71
\\
\hline
15
&
$\ssb$
&
$\ssb$
\\
\hline
14
&
$\begin{array}{c}
\ssb \ssb \\
b_{70} \ssc \ssc c_{70} 
\end{array}$
&
$\begin{array}{c}
\ssb \ssb \\
f_{71} \ssc \phantom{f_{71}}
\end{array}$
\\
\hline
13
&
$\begin{array}{c}
 \ssb  \\
 \ssc \ssc  \\
a_{70} \sscc \phantom{a_{70}}
\end{array}$
&
$\begin{array}{c}
d_{71} \ssc \ssc \ssc e_{71} \\
b_{71} \sscc \sscc c_{71} 
\end{array}$
\\
\hline
12
&
$\begin{array}{c}
\sscc \\
*
\end{array}$
&
$\begin{array}{c}
a_{71} \sscc \sscc \sscc \sscc \sscc \phantom{a_{71}} \\
G_{71} * \phantom{G_{71}}
\end{array}$
\end{tabular}
\caption{The algebraic $\tmf$-resolution for $\Ext^{*,*}_{A_*}(H(1,4)
\otimes DH(1))$ in the vicinity of $(t-s,s) = (71, 12)$}\label{tab:v_2^32b}
\end{table}

The element $v_2^{32} \in \Ext^{32,224}_{A(2)_*}(H(1,4))$ detects a non-trivial
permanent cycle of order $2$ in $\tmf_{192}M(1,4)$.  We deduce that the
element
$$ v_2^{32} \in \Ext_{A(2)_*}(H(1,4) \otimes DH(1)) $$
is a permanent cycle in the MASS for $\tmf \wedge M(1,4) \wedge DM(1)$.
Consider the map of MASSs
\begin{equation}\label{eq:MASSmap}
\xymatrix{
\Ext_{A_*}^{s,t}(H(1,4) \otimes DH(1)) \ar@{=>}[r] \ar[d] &
\pi_{t-s}(M(1,4) \wedge DM(1)) \ar[d] \\
\Ext_{A(2)_*}^{s,t}(H(1,4) \otimes DH(1)) \ar@{=>}[r] &
\tmf_{t-s}(M(1,4) \wedge DM(1))
}
\end{equation}
induced by the map
$$ M(1,4) \wedge DM(1) \rightarrow \tmf \wedge M(1,4) \wedge DM(1). $$
Because $v_2^{32}$ is a permanent cycle in the MASS for $\tmf \wedge M(1,4)
\wedge DM(1)$, we deduce that in the MASS for $M(1,4) \wedge DM(1)$, the
differential $d_r(v_2^{32})$ cannot hit an element coming from
$\Ext_{A(2)_*}$ in the algebraic $\tmf$-resolution (these elements are
represented by a $\ssb$ in Table~\ref{tab:v_2^32a}).  Thus the only
possible targets for $d_r(v_2^{32})$ in Table~\ref{tab:v_2^32a} are
\begin{equation}\label{eq:targets}
g^6 a_{71}, g^6 b_{71}, g^6 c_{71}, g^6 d_{71}, g^6 e_{71}, g^6 f_{71}
\end{equation}
or an element of the group $G_{191}$.
We claim that none of these elements persist to detect a non-trivial
element of the $E_4$-page of the MASS.

Each of the elements in (\ref{eq:targets}) is in the image of
multiplication by $g^6$.  Because the groups $G_{71}$ and $G_{191}$ lie on the 
edge of
the slope $1/5$ vanishing line of Lemmas~\ref{lem:vanishing1} and
\ref{lem:vanishing2}, each of the elements in $G_{191}$ are of the form
$g^6y$ for $y \in G_{71}$.

Suppose that $x$ is a linear combination of the elements
\begin{equation}\label{eq:pretargets}
a_{71}, b_{71}, c_{71}, d_{71}, e_{71}, f_{71}
\end{equation}
and the elements in $G_{71}$.  We must show that $g^6 x$ cannot be
the non-trivial image of $d_r(v_2^{32})$ for $r \ge 4$.

If $x$ is a $d_r$-cycle for $r \le 3$, then
$x$ persists to $E_4$.  Using the multiplicative structure of the MASS
(Proposition~\ref{prop:MASSmult}) together with the fact that 
$g^6 = 0$
in
the $E_4$-page of the MASS for $M(1,4)\wedge DM(1,4)$
(Theorem~\ref{thm:diff})\footnote{
This statement must be interpreted with care --- Theorem~\ref{thm:diff}
asserts that there is an element in $E_2$ of the MASS for $M(1,4) \wedge
DM(1,4)$ which is detected by $\td{g}^6$ in the algebraic
$\tmf$-resolution, and which is the target of a $d_3$ in the MASS.
} 
, we deduce that $g^6 x$ is 
zero 
in $E_4$.  It
therefore cannot be a non-trivial target for $d_r(v_2^{32})$.

Suppose, however, that $d_r(x)$ is non-trivial for some $r \le 3$.  
Since differentials in the algebraic $\tmf$ resolution
must increase filtration, we deduce that the only possible targets for
$d_r(x)$ are linear combinations of 
$$ a_{70}, b_{70}, c_{70} $$
and $\ssb$'s in Table~\ref{tab:v_2^32b} for which $t-s = 70$ and $s \ge
14$.  However, each of these $\ssb$'s 
map to non-trivial permanent cycles under the map of spectral sequences
(\ref{eq:MASSmap}), and therefore cannot be the target of MASS
differentials.  The only remaining possibilities are
\begin{align*}
\text{Case (1):} \qquad & d_1(x) = a_{70}, \\
\text{Case (2):} \qquad & d_2(x) = t_1 b_{70} + t_2 c_{70},
\end{align*}
for $(0,0) \ne (t_1, t_2) \in \FF_2 \oplus \FF_2$.  Using
Theorem~\ref{thm:diff} we see that in these cases we would
respectively have:
\begin{align*}
\text{Case (1):} \qquad & d_1(g^6x) = g^6a_{70}, \\
\text{Case (2):} \qquad & d_2(g^6x) = t_1 g^6b_{70} + t_2 g^6c_{70}.
\end{align*}
If we are in Case (1), we are done: the differential $d_r(v_2^{32})$ cannot
be detected by $g^6 x$ because $g^6 x$ does not persist to $E_2$.  If we
are in Case (2), however, 
we must verify that $t_1 g^6b_{70} + t_2 g^6c_{70}$ is not
in the image of a $d_1$-differential.  The only possibility is
\begin{equation}\label{eq:possibility}
d_1(s_1 v_2^8 k_{143} + s_2 v_2^8 l_{143}) = t_1 g^6b_{70} + t_2 g^6c_{70}.
\end{equation}
The algebraic $\tmf$-resolution for $\Ext_{A_*}(H(1,4)\otimes DH(1))$ in
the vicinity of the elements $k_{143}$ and $l_{143}$ is displayed below.
\vspace{10pt}

\centerline{
\begin{tabular}{c|c|c}
$s \backslash t-s $ & 142 & 143 
\\
\hline
30 & $\ssb$ & $\ssb$ 
\\
\hline
29 & $\ssb \ssb$ & $k_{143} \ssb \ssb l_{143}$
\end{tabular}
}
\vspace{10pt}

We see that $k_{143}$ and $l_{143}$ must be $d_1$-cycles.  By
Proposition~\ref{prop:v_2^8} and Lemma~\ref{lem:resolutionmult}, we deduce
that $v_2^8k_{143}$ and $v_2^8l_{143}$ must be $d_1$-cycles.  Thus
Possibility~(\ref{eq:possibility}) cannot occur, and we deduce that
in Case (2), $d_2(g^6 x)$ does not vanish.  
We conclude that in Case (2), $g^6 x$ cannot persist to
$E_4$ and therefore it cannot be the
target of $d_r(v_2^{32})$ for $r \ge 4$.
\end{proof}

\nocite{*}
\bibliography{v2_32}

\end{document}